\documentclass{birkjour}
\usepackage{multicol}
\usepackage{enumerate,amsmath,amssymb,graphicx}
\usepackage{subfig,color}
\def\RR{\mathbb R}

\graphicspath{{./}{figures/}}

\newtheorem{theorem}{Theorem}
\newtheorem{proposition}{Proposition}
\newtheorem{lemma}{Lemma}

\newtheorem{remark}{Remark}

\begin{document}

\title[Dynamic programming for controlled fractional SIS models]{A dynamic programming approach for controlled fractional SIS models}

\keywords{Fractional SIS models, Caputo-Fabrizio operator, optimal control, dynamic programming, asymptotic solutions.}
\subjclass{26A33, 92D30, 49J20,\\49L25, 65M22}
\author[S. Cacace]{Simone Cacace}
\address{S. Cacace, Dipartimento di Matematica e Fisica, Università degli Studi di Roma Tre, Largo San Leonardo Murialdo, 1, 00146, Roma, Italy}
\email{cacace@mat.uniroma3.it}

\author[A.C. Lai]{Anna Chiara Lai}
\address{A. C, Lai, Dipartimento di Scienze di Base e Applicate per l'Ingegneria,
Sapienza Universit\`a di Roma, Via Antonio Scarpa 10, 00161, Roma, Italy}
\email{annachiara.lai@uniroma1.it}

\author[P. Loreti]{Paola Loreti}
\address{P. Loreti, Dipartimento di Scienze di Base e Applicate per l'Ingegneria,
Sapienza Universit\`a di Roma, Via Antonio Scarpa 16, 00161, Roma, Italy}
\email{paola.loreti@uniroma1.it}
\begin{abstract}We investigate a susceptible-infected-susceptible (SIS) epidemic model based on the Caputo-Fabrizio operator. After performing an asymptotic analysis of the system, we study a related finite horizon  optimal control problem with state constraints. We prove that the corresponding value function is a viscosity solution of a dynamic programming equation. We then turn to the asymptotic behavior of the value function, proving its convergence to the solution of a stationary problem, as the planning horizon tends to infinity. Finally, we present some numerical simulations providing a qualitative description of the optimal dynamics and the value functions involved.  \end{abstract}
\maketitle{}
\section{Introduction}
Consider the Caputo-Fabrizio derivative of order $\alpha\in[0,1]$ for a function  $f\in H^1((a,b))$, $a<b$ 
$$D^{\text{CF}}_\alpha f(t):=\frac{M(\alpha)}{1-\alpha}\int_a^t f'(\tau) e^{-\frac{\alpha}{1-\alpha}(t-\tau)}d\tau\,,$$
where $M(\alpha)$ is a scaling factor satisfying $M(0)=M(1)=1$.

\medskip\noindent
We study the fractionary SIS system
\begin{equation}\label{SIS}
\begin{cases}
D^{\text{CF}}_\alpha S= -\left(\frac{\beta}{N} S-\gamma\right)I\\
D^{\text{CF}}_\alpha I = \left(\frac{\beta}{N} S-\gamma\right)I\\
S(0)=S_0\\
I(0)=I_0\,,
\end{cases}
\end{equation}
where $\beta,\gamma, N>0$, and  $I_0,S_0\geq 0$ satisfy $I_0+S_0=N$. Moreover, we make the technical assumption
\begin{equation}\label{gammacond}
\alpha+(1-\alpha)(\beta-\gamma)\leq M(\alpha).
\end{equation}
Note that above condition is trivially satisfied if $\alpha=1$, since $M(1)=1$ by definition. 
The functions $S(t)$ and $I(t)$ represent the size of susceptible and infective individuals, respectively. $N$ is assumed to be the size of the total population at initial time.   $\beta$ is the average number of contacts per person per time, multiplied by the probability of disease transmission in a contact between a susceptible and an infectious subject, and $\gamma$ is the recovery rate. 

{Fractional epidemic models attracted the interest of researchers, see \cite{chen2021AMM} for a recent review, due to the possibility of tuning the derivative order $\alpha$ for applications to real data fitting, see for instance \cite{li2019NA}. An important feature is the ability to incorporate  memory effects into the model: in particular, we refer  \cite{SISCF} (and the references therein) for a detailed study of the inclusion of memory in epidemic models by using the Caputo-Fabrizio operator.  The peculiarity of this fractional operator, introduced in \cite{CapFab}, is the presence of a non-singular, exponential kernel.}
We show below that  \eqref{SIS} rewrites as 
\begin{equation}\label{saturatedSIS}
\begin{cases}S'=-\frac{\lambda_{\alpha}}{1+k_{\alpha} I}SI+\frac{r_{\alpha}}{1+k_{\alpha} I} I\\
I'=\frac{\lambda}{1+k_{\alpha} I}SI-\frac{r_{\alpha}}{1+k_{\alpha} I} I
\end{cases}\end{equation}
where 
$$\lambda_{\alpha}:=\frac{\beta\alpha}{N(M(\alpha)-(1-\alpha)(\beta-\gamma)},\quad r_{\alpha}:=\frac{\gamma\alpha}{N(M(\alpha)-(1-\alpha)(\beta-\gamma)},$$
and
$$ k_{\alpha}:=\frac{(1-\alpha)2\beta}{N(M(\alpha)-(1-\alpha)-(1-\alpha)(\beta-\gamma)}.$$
{ Dynamics as \eqref{saturatedSIS} underlie the class of SIS models with  \emph{saturated incidence rate} $H_{\alpha}(I):= \frac{\lambda_{\alpha} I}{1+k_{\alpha} I}$, and \emph{saturated treatment function} $T_{\alpha}(I):=\frac{r_{\alpha} I}{1+k_{\alpha} I}$, see  \cite{saturatedSIS,saturatedsis2}. The system \eqref{saturatedSIS} models a stable population in which the natural recovery rate is zero}:  individuals recover from the disease only if they are treated, and they are healed at the rate $T_\alpha(I)$. In particular, $r$ is the cure rate, whereas $1/(1+k_{\alpha} I)$ measures the reverse effect of the delayed treatment, due, for instance, to the limited health system capacity. Also the incidence rate $H_\alpha(I)$ is assumed to saturate as $I$ increases: this models the psychological effect of the awareness in the susceptible population about the existence of a large size of infected individuals, inducing a more cautious behavior \cite{capasso1978}. It is worth noting that \eqref{saturatedSIS} is  also related to equations arising in the dynamics between tumor cells, immune-effector cells, and immunotherapy \cite{tumor-immune}. We refer to \cite{piccoli} and the references therein, for an application of  control theory to the optimization of cancer therapies in a general, nonlinear setting. 

\medskip Using the conservation of the population, i.e., the identity $S=N-I$, the system \eqref{SIS} reduces to a non-linear ordinary differential equation depending on the infected population $I(t)$ only, see Theorem \ref{p2}. We then complete this equation with a linear control term, and we address the problem of minimizing the size of infected individuals plus a quadratic cost on the control. To this end, we study the associated dynamic programming equation, namely an evolutive Hamilton-Jacobi equation for which the value function $u^\alpha(x,t)$ is proved to be a viscosity solution, see Theorem \ref{thmexistence}. Our main result is Theorem \ref{thm1}, in which we prove, as $t\to+\infty$, the convergence of $u^\alpha(x,t)$ to  the value function $v^\alpha(x)$ of an associated stationary problem.
Theorem \ref{thmexistence} and Theorem \ref{thm1}, and the techniques adopted for their proofs, are inspired by the paper \cite{FIL06}. For a general introduction on the topic, we refer to \cite{lions}. 

Finally, we introduce a suitable finite difference scheme for solving the Hamilton-Jacobi equation and building the corresponding optimal trajectories. Some numerical tests complete the presentation, validating our results and providing a qualitative analysis of the solutions. 
\section{The SIS model with Caputo Fabrizio derivative}
We begin by noting the indentity for $f\in H^1((a,b))$:
\begin{equation}\label{dD}
\frac{d}{dt} D^{\text{CF}}_\alpha f(t)=\frac{M(\alpha)}{1-\alpha}f'(t)-\frac{\alpha}{1-\alpha}D^{CF}_\alpha f(t)\quad \forall t\in(a,b),\alpha\in[0,1).
\end{equation}
We have the following result, relating \eqref{SIS} to an algebraic identity and to an ordinary differential equation. 

\begin{theorem}\label{p2}
Assume \eqref{gammacond}. Let $I_0,S_0>0$, $N:=S_0+I_0$, and $(S,I)$ be a solution of \eqref{SIS}. Then $S(t)=N-I(t)$ for all $t\geq 0$, and $I$ is the unique,  global positive solution of the Cauchy problem
\begin{equation}\label{ordSIS}
\begin{cases} 
I'=b_\alpha(I):=(\beta-\gamma-\frac{\beta}{N}I)I \dfrac{\alpha}{M(\alpha)-(1-\alpha)(\beta-\gamma-\frac{2\beta}{N}I)}\\
I(0)=I_0.
\end{cases}
\end{equation}
\end{theorem}
\begin{proof}
Fix $I_0,S_0>0$, $N:=S_0+I_0$, and  a solution $(S,I)$ of \eqref{SIS} with initial datum $(S_0,I_0)$. Define $N(t):=S(t)+I(t)$. By the linearity of the Caputo-Fabrizio operator, summing the two equations in \eqref{SIS} we get $D_\alpha^{\text{CF}}N(t)\equiv 0$. Then, also in view of \eqref{dD},
$$0=\frac{d}{dt}D_\alpha^{\text{CF}}N(t)=\frac{M(\alpha)}{1-\alpha}N'(t)\,,$$
from which we deduce $N(t)=N(0)=S(0)+I(0)=N$.
Replacing $S=N-I$ in the second equation of \eqref{SIS}, we have
\begin{align*}
\frac{d}{dt}D^{\text{CF}}_\alpha I& =\frac{d}{dt} \left(\left(\frac{\beta}{N} S-\gamma\right)I\right)\\
&=\frac{d}{dt} \left(\left(\beta-\gamma-\frac{\beta}{N} I\right)I\right)= \left(\beta-\gamma-\frac{2\beta}{N} I\right)I'\,.
\end{align*}
On the other hand, using \eqref{dD} and the second equation in \eqref{SIS}, we get
\begin{align*}
\frac{d}{dt}D^{\text{CF}}_\alpha I
&=\frac{M(\alpha)}{1-\alpha}I'-\frac{\alpha}{1-\alpha}\left(\beta-\gamma-\frac{\beta}{N} I\right)I.
\end{align*}
Therefore
$$\left(\beta-\gamma-\frac{2\beta}{N} I\right)I'=\frac{M(\alpha)}{1-\alpha}I'-\frac{\alpha}{1-\alpha}\left(\beta-\gamma-\frac{\beta}{N} I\right)I.$$
Making the above equation explicit with respect to $I'$, we obtain the second equation of \eqref{ordSIS}.

Now,   $b_\alpha$ is well defined in $(x_0,+\infty)$, where 
$$x_0:=((\beta-\gamma)(1-\alpha)-M(\alpha))\frac{2N}{\beta(1-\alpha)}$$
and assumption \eqref{gammacond} implies that $x_0<0$. 
Moreover,  for $\alpha\in[0,1)$, we get  
$$b_\alpha'(I)=\frac{\alpha  \left(\beta -\gamma -\frac{2 \beta 
   I}{N}\right)}{M(\alpha)-(1-\alpha ) \left(\beta -\gamma -\frac{2 \beta 
   I}{N}\right)}-\frac{2 (1-\alpha ) \alpha  \frac{\beta}{N}  I \left(\beta -\gamma
   -\frac{\beta  I}{N}\right)}{ \left( M(\alpha)-(1-\alpha ) \left(\beta
   -\gamma -\frac{2 \beta  I}{N}\right)\right)^2}$$
and 
  $$\lim_{I\to+\infty}b'_\alpha(I)\to-\frac{\alpha}{2(1-\alpha)} =-\frac{\alpha}{2(1-\alpha)}.$$
  This, together with the continuity of $b'_\alpha$ in $(x_0,+\infty)$, implies that $b'_\alpha$ is uniformly bounded in any closed subset of $(x_0,+\infty)$, in particular in $[x_1,+\infty)$ for some fixed $x_1\in(x_0,0)$. We conclude that $b_\alpha$ is Lipschitz continuous in $(x_1,+\infty)$, hence the second equation of \eqref{ordSIS} admits a unique, global solution in $(x_1,+\infty)$. 
It is left to show that 
if $I_0>0$ then $I(t)\geq0$ for all $t>0$. We argue by contradiction. Assume that $I_0>0$ and that the corresponding solution satisfies $I(\bar t)<0$ for some $\bar t>0$. Then, by continuity, $I(t^*)=0$ for some $t^*\in(0,\bar t)$, so that  $\bar I(t):=I(t+t^*)$ is the solution of \eqref{ordSIS} with initial datum $0$. On the other hand, $b_\alpha(0)=0$ implies that $I(t)\equiv 0$ is the unique solution of \eqref{ordSIS}, with initial datum $0$, and  consequently $\bar I(t)\equiv 0$. It follows that $0=\bar I(\bar t-t^*)=I(\bar t)<0$, namely the required contradiction.
 
In the remaining case $\alpha=1$, we recover the classical logistic equation
$$I'= \left(\beta-\gamma-\frac{\beta}{N}I\right)I\,,$$
whose unique solution in $[0,+\infty)$ is explicitly given by $$I(t)=\frac{NI_0 (\beta-\gamma)}{e^{t (\gamma-\beta)} (N(\beta-\gamma)- \beta I_0)+\beta I_0}.$$
and satisfies $I(t)\ge 0$ for all $t>0$ if $I_0>0$.
\end{proof}
\subsection{Equlibria and asymptotic behavior}
We introduce the reproduction factor $\rho:=\beta/\gamma$ and we describe the asymptotic behavior of \eqref{ordSIS} according to the cases of $\rho>1$ and $\rho\leq 1$. 
\begin{proposition}\label{equilibrium}
If $\rho>1$ (respectively, $\rho\leq 1$) then the \emph{endemic population} 
$E:=N(1-\frac{1}{\rho})$ (resp, the equilibrium $0$) is an asymptotically stable equilibrium point for \eqref{ordSIS}.
 In particular, for $I_0>0$, the corresponding solution $I$ 
 satisfies $I(t)\to E$ (resp. $I(t)\to 0$) as $t\to+ \infty$. 
 \end{proposition}
\begin{proof}
We rewrite the second equation in \eqref{ordSIS} in terms of $\rho$:
\begin{equation}\label{RordSIS}
\begin{cases} 
I'=(\rho-1- \frac{\rho}{N}I)I \dfrac{\alpha}{M(\alpha)/\gamma-(1-\alpha)(\rho-1-\frac{2\rho}{N}I)}\\
I(0)=I_0.
\end{cases}
\end{equation}
We set $I_e:=E$ if $\rho>1$, $I_e:=0$ otherwise, and we define $V(x):=\frac{1}{2}(x-I_e)^2$. We prove that $V$ is a Lyapunov function for \eqref{RordSIS}. To this end, note that $V$ is a positive definite function in $\RR\setminus\{I_e\}$ and that $V$ is smooth in $\RR$. Moreover, by a direct computation, we get
$$V'(x)b_\alpha(x)=
\begin{cases}
 \dfrac{-\alpha N(\rho-1- \frac{\rho}{N}x)^2x}{M(\alpha)/\gamma-(1-\alpha)(\rho-1-\frac{2\rho}{N}x)}<0 \quad &\text{if } \rho>1,\, x>0\\\\
 \dfrac{\alpha(\rho-1- \frac{\rho}{N}x)x^2}{M(\alpha)/\gamma-(1-\alpha)(\rho-1-\frac{2\rho}{N}x)}<0 \quad &\text{if } \rho\leq 1,\, x>0\,,\\
\end{cases}
$$
and this concludes the proof. 
\end{proof}
\begin{remark} This result was earlier proved in \cite{saturatedSIS} in the framework of saturated SIS models, here we propose an alternative proof for the fractional context under exam. 
\end{remark}

\section{An optimal control problem for fractional SIS}\label{s3}
Fix $\beta,\gamma,N>0$, $\rho=\beta/\gamma$ as before.
We consider 
a controlled version of the infected population dynamics \eqref{ordSIS}: 
\begin{equation}\label{controlya}
\begin{cases} I'(t)=b_\alpha(I(t))+\xi(t)\quad \text{for }t\in(0,T)\\
I(0)=x\,,\end{cases}\end{equation}
where $x\geq 0$ and $\xi\in L^1((0,T))$ is the control function. We consider the following finite horizon optimal control problem:  minimize with respect to $\xi$  
\begin{equation}\label{pb}\begin{split}\text{} \int_0^T \frac{I^2(t)}{2}+\frac{\xi^2(t)}{2}dt+\phi(I(T))
\quad \text{subject to \eqref{controlya} and } I(t)\geq 0, t\in(0,T)\,,\end{split}\end{equation}
where $T>0$, $[0,T]$ is the given planning horizon, and $\phi(I(T))$ is the expected future cost,  depending on the size of the individuals remaining infected at time $T$. We assume that $\phi$ is a continuous, nonnegative function attaining its global minimum at $x=0$. The state constraint $I\geq 0$ is introduced for modeling reasons. 

We prove that the value function $u^\alpha(x,T)$ associated to the problem \eqref{pb}  is a viscosity solution of a related dynamic programming equation.  

To this end, we denote by $I(t;\xi,x)$ the absolutely continuous solution of \eqref{controlya}. We say that a trajectory-control pair $(I(t;\xi,x),\xi(t))$ is admissible if $\xi\in L^1((0,T))$ and $I(t;\xi,x)\geq 0$ for all $t\in (0,T)$, and we denote by $\mathcal A_T\subset L^1((0,T))$ the set of admissible controls. Then, we define the value function associated to \eqref{pb}
\begin{equation}\label{value}
u^\alpha(x,T):=\inf_{\xi \in \mathcal  A_T}\int_0^T \frac{I^2(t;\xi,x)}{2}+\frac{\xi^2(t)}{2}dt+\phi(I(T;\xi,x))
\end{equation}

\noindent and we consider the following Hamilton-Jacobi equation  \begin{equation}\label{hj}
\begin{cases}
u_t-b_\alpha(x)Du+\frac{1}{2}Du^2-\frac{1}{2}x^2=0& \text{in } (0,+\infty)\times(0,+\infty)\\
u|_{t=0}=\phi(x)& x\in(0,+\infty)\,,
\end{cases}
\end{equation}
{in which the Hamiltonian is provided by the Legendre transform  
$$\sup_{\xi\in \RR} \left\{-(b_\alpha(x)+\xi) Du-\frac{1}{2}x^2-\frac{1}{2}\xi^2\right\}\,,$$
where the supremum is achieved by the optimal control $\xi=-Du$.}
\begin{remark}
Equation \eqref{hj} is the dynamic programming equation associated to the problem of minimizing with respect to $\xi$ 
\begin{equation*}\int_0^T \frac{I^2(t)}{2}+\frac{\xi^2(t)}{2}dt+\phi(I(T))
\qquad \text{subject to \eqref{controlya}}\,,\end{equation*}
namely a version of \eqref{pb} in which the state constraint $I(t)\geq 0$ is removed. The possibility to  neglect the state constraint, combined with  the unboundedness of the control set and the nonlinearity of the drift $b_\alpha$ (which is not Lipschitz when $\alpha=1$) are elements of novelty that require an ad hoc analysis.
\end{remark}

In agreement with the classical theory, we have the following result.
\begin{theorem}\label{thmexistence}
For all $\alpha\in(0,1]$, the value function $u^\alpha$ defined in \eqref{value}
is a viscosity solution of \eqref{hj}. 
\end{theorem}
\begin{proof} 
The proof  is based on showing that $u^{\alpha}$ simultaneously fulfills the definition of viscosity sub-solution and of viscosity super-solution. Technical computations can be easily adapted from  \cite[Theorem 10]{FIL06} (see also the proof of Theorem \ref{thmva} below) and we omit them for brevity. The main ingredients are the Dynamic Programming Principle (Proposition \ref{DPP}) and the continuity of $u^{\alpha}$  (Proposition \ref{pc}), that are proved in detail the following subsections. 
\end{proof}

\subsection{Auxiliary results for the proof of Theorem \ref{thmexistence}}\label{thmexp}
 Let $\RR_+:=(0,+\infty)$. For $(x,T) \in [0,+\infty)\times\RR_+$, let $\mathcal C^+(x,T)$  denote the space of non-negative, absolutely continuous functions  $X: [0,T]\to[0,+\infty)$  such that $X(0)=x$. Then, the value function $u^\alpha$ defined in \eqref{value} can be rewritten in a form suitable for our purposes: 
\begin{equation}\label{udef+}
\begin{split}
    u^\alpha(x,T)=\inf\left\{\int_0^T\frac{1}{2}X(t)^2+ \frac{1}{2}(b_\alpha(X(t))-\dot X(t))^2+\phi(x(T))\right.\\\mid X\in \mathcal C^+(x,T)\Bigg\}.
\end{split}\end{equation}

The proof of Theorem \ref{thmexistence} relies on several preliminary results, that can be summarized in a set of estimates for $u^\alpha$, proved in Section \ref{estss}, the Dynamic Programming Principle, and a continuity result for $u^\alpha$, proved in Section \ref{conss} below. 

\subsubsection{Estimates for $u^\alpha$}\label{estss} For $R>0$, we set
\begin{align}&\label{C1} C_1^\alpha(R):=(R+||b_\alpha||_{L^\infty([0,R])})^2+||\phi||_{L^\infty([0,R])}\,, \\
& \label{C2} C_2^\alpha(R):=\sup\left\{y\in(0,+\infty)\mid \hat B_\alpha(y)\leq ||\hat B_\alpha||_{L^\infty([0,R])}+1+C^\alpha_1(R)\right\}\,,\end{align}
where 
\begin{equation}\label{Bdef}
\hat B_\alpha(x):=\int_0^x -b_\alpha(s)ds.\end{equation}
 Note that, since $\hat B_\alpha(0)=0$, $\hat B_\alpha$ is continuous,  and $\hat B_\alpha(x)\to+\infty$ as $x\to+\infty$, then $C_2^\alpha(R)\in(0,+\infty)$ for all $R>0$. In particular $C_2^\alpha(R)\geq R$. Indeed, clearly $\hat B_\alpha(R)\leq ||\hat B_\alpha||_{L^\infty([0,R])}$, hence $R\in \{y\mid \hat B_\alpha(y)\leq ||\hat B_\alpha||_{L^\infty([0,R])}+1+C^\alpha_1(R), y\in(0,+\infty)\}$ and, consequently $R\leq C_2^\alpha(R)$.  
Finally define
 \begin{equation}C_3^\alpha(R):=R^2+||b_\alpha||_{L^\infty([0,R])}^2.\end{equation}
We remark that $C_1^\alpha(R)$ and $C_2^\alpha(R)$ depend on $\phi$ only via $||\phi||_{L^\infty([0,R])}$, whereas $C_3^\alpha(R)$ is independent from $\phi$.

\medskip
The next three lemmas provide estimates for $u^\alpha$ related to $C_1^\alpha$, $C_2^\alpha$ and $C_3^\alpha$. 
\begin{lemma}\label{l28}
For each $\alpha\in[0,1]$ and $R>0$ 
$$\phi(0)\leq u^\alpha(x,T)\leq C_1^\alpha(R) \quad \text{ for } (x,T)\in [0,R]\times(0,\infty).$$ 
In particular $u^\alpha(0,T)=\phi(0)$, i.e., for all $T>0$, $u^\alpha(\cdot,T)$ attains its global minimum at $x=0$. 
\end{lemma}
\begin{proof}
The lower estimate readily follows by the assumption that $\phi$ attains its global minimum at $x=0$. Indeed, for all $T\geq0$ and $\varepsilon>0$ there exists $X\in \mathcal C^+(x,T)$ such that 
\begin{align*}u^\alpha(x,T)+\varepsilon&> \int_0^T \frac{1}{2}X(t)^2+\frac{1}{2}(b_\alpha(X(t))-\dot X(t))^2dt+\phi(X(T))\\&\geq \phi(X(T))\geq \phi(0).\end{align*}
Therefore $u^\alpha(x,T)\geq \phi(0)$ for all $x,T\geq 0$. Moreover, choosing $X(t)\equiv 0$, we have for all $T>0$
$$u^\alpha(0,T)\leq \int_0^T \frac{1}{2}X(t)^2+\frac{1}{2}(b_\alpha(X(t))-\dot X(t))^2ds+\phi(X(T))=\phi(0).$$ 
Hence $\phi(0)=u(0,T)\leq u(x,T)$ for all $x\geq 0$, $T>0$. 

To prove the upper estimate, consider the curve $X\in \mathcal C^+(x,T)$ defined by
$$X(t)=\begin{cases}
x-tx &\text{for } 0\leq t\leq 1\\
0 & \text{ for } t>1
\end{cases}$$
so that $0\leq X(t)\leq x\leq R$ and $\dot X(t)=-x$ for $t\in [0,1]$. 
Since $b_\alpha(0)=0$, then  $T\geq 1$ implies
$$\int_0^T \frac{1}{2}X(t)^2+\frac{1}{2}(b_\alpha(X(t))-\dot X(t))^2)dt=\int_0^1\frac{1}{2}X(t)^2+\frac{1}{2}(b_\alpha(X(t))-\dot X(t))^2)dt$$
and $\phi(X(T))=\phi(0)$.
For a general $T>0$ we then have 
\begin{align*}
u^\alpha(x,T)\leq& \int_0^T \frac{1}{2}X(t)^2+\frac{1}{2}(b_\alpha(X(t))-\dot X(t))^2dt+\phi(x(T))\\
\leq&\int_0^1 \frac{1}{2}X(t)^2+\frac{1}{2}(b_\alpha(X(t))-\dot X(t))^2dt+||\phi||_{L^\infty([0,R])}\\
=&\int_0^1\frac{1}{2} x^2(1-t)^2+\frac{1}{2}(b_\alpha(x(1-t))+x)^2dt+||\phi||_{L^\infty([0,R])}\\
\leq&\int_0^1\frac{1}{2} R^2+\frac{1}{2}(||b_\alpha||_{L^\infty([0,R])}+R)^2dt+||\phi||_{L^\infty([0,R])}\\
=& \,R^2+\frac{1}{2}||b_\alpha||_{L^\infty([0,R])}^2+R||b_\alpha||_{L^\infty([0,R])}+||\phi||_{L^\infty([0,R])}\\
<&\, C^\alpha_1(R).
\end{align*}\end{proof}

 \begin{lemma}\label{l29}Fix $\alpha\in[0,1]$. For each $R>0$, if  $(x,T)\in [0,R]\times \RR_+$,  $X(t)\in \mathcal C^+(x,T)$ and
\begin{equation}\label{cond}u^\alpha(x,T)+1\geq \int_0^T \frac{1}{2}X(t)^2+\frac{1}{2}(b_\alpha(X(t))-\dot X(t))^2)dt+\phi(X(T)),\end{equation}
then for $t\in(0,T]$
$$|X(t)|\leq C_2^\alpha(R).$$
\end{lemma}
\begin{proof}
Let $X\in\mathcal C^+(x,T)$ satisfy \eqref{cond}. Then $X(t)\geq 0$ by the definition of $\mathcal C^+(x,T)$. To prove $X(t)\leq C_2^\alpha(R)$, note that by Lemma \ref{l28}
$$
\int_0^T \frac{1}{2}X(t)^2+\frac{1}{2}(b_\alpha(X(t))-\dot X(t))^2dt+\phi(X(T))\leq 1+C^\alpha_1(R).$$ 
In particular, also recalling that we assumed $\phi$ to be non negative, for all $\tau \in [0,T]$
$$
\int_0^{\tau} \frac{1}{2}X(t)^2+\frac{1}{2}(b_\alpha(X(t))-\dot X(t))^2)dt\leq 1+C^\alpha_1(R).$$ 
From this we deduce
$$
\int_0^\tau-b_\alpha(X(t))\dot X(t)dt\leq 1+C_1(R).$$ 
Integrating the left handside of above expression, we get
\begin{equation}\label{Balpha}
\hat B_\alpha(X(\tau))\leq \hat B_\alpha(x)+1+C^\alpha_1(R)\leq ||\hat B_\alpha||_{L^\infty([0,R])}+1+C^\alpha_1(R).\end{equation}
Since $X(\tau)\geq 0$, we deduce the claimed inequality $X(\tau)\leq C^\alpha_2(R)$ for all $\tau\in[0,T]$. 
\end{proof}
\begin{lemma}\label{l42}
For all $\alpha\in[0,1]$ and for all  $(x,T)\in [0,+\infty)\times \RR_+$ it holds
$$u^\alpha(x,T)\leq \phi(x)+\frac{1}{2}(x^2+(b_\alpha(x))^2)T.$$
\end{lemma}
\begin{proof}
Choose $X(t)\equiv x$ and remark that
\begin{align*}u^\alpha(x,T)&\leq \int_0^T \frac{1}{2}X(t)^2+\frac{1}{2}(b_\alpha(X(t))-\dot X(t))^2dt+\phi(x(T))\\
&=\frac{1}{2}(x^2+b_\alpha(x)^2)T+\phi(x).\end{align*}
\end{proof}
\subsubsection{Continuity of $u^\alpha$}\label{conss} The next two results investigate the  dependence of $u^\alpha$ on the initial datum $\phi$ and the local uniform continuity of $u^\alpha$ with respect to the space variable $x$, respectively. 
For any continuous function 
$f:{D}\subseteq [0,+\infty)\to \RR$ we call \emph{modulus (of continuity)} of $f$ any increasing, continuous function $\omega : [0, \infty)\times[0,\infty)\to [0, \infty)$ such that $\omega(0) = 0$, $\omega(r) > 0$ for every $r > 0$ and $|f(x_1) - f(x_2)| \leq  \omega(|x_1 - x_2|)$ for all $x_1, x_2\in  {D}$. We denote by $\omega_{\phi,R}$ the modulus of continuity of $\phi$ restricted to $[0,R]$.

\begin{lemma} \label{l456} Let $\alpha\in[0,1]$ and $R>0$. Define  
$$C^\alpha(R):=\max\{C^\alpha_1(R),||b_\alpha||_{L^\infty([0,C^\alpha_2(R)])}\}.$$ 

For all $(x,T)\in [0,R]\times\RR_+$ if $X\in \mathcal C^+(x,T)$ and
$$u^\alpha(x,T)+1\geq  \int_0^T \frac{1}{2}X(t)^2+\frac{1}{2}(b_\alpha(X(t))-\dot X(t))^2dt+\phi(x(T)),$$
then
\begin{equation}\label{sigma}|X(t)-x|\leq \sigma_R(t):=tC^\alpha(R)+\sqrt{t}(C^\alpha(R)+2)\quad \text{for } t\in[0,T].\end{equation}
Moreover,  for all $(x,t)\in [0,R]\times\RR_+$
\begin{equation}\label{nu}|u^\alpha(x,t)-\phi(x)|\leq \nu_R(t):=\max\{C_3^\alpha(R)t,\omega_{\phi,R}( \sigma_R(t))\}.\end{equation}
In particular $\nu_R$ depends on $\phi$ only via $\omega_{\phi,R}$ and $||\phi||_{L^\infty([0,R])}$.
\end{lemma}
\begin{proof}
Assume that  $X\in \mathcal C^+(x,T)$ satisfies
$$u^\alpha(x,T)+1\geq  \int_0^T \frac{1}{2}X(t)^2+\frac{1}{2}(b_\alpha(X(t))-\dot X(t))^2dt+\phi(x(T)).$$
By Lemma \ref{l28} and Lemma \ref{l29}, we respectively get
\begin{align*}&u^\alpha(x,T)\leq C^\alpha_1(R)\leq C^\alpha(R) \qquad \text{for } x\in [0,R]\times (0,+\infty)\,,\\
&|X(t)|\leq C_2^\alpha(R)\qquad \text{for } t\geq 0.\end{align*}
Moreover, it follows from the latter inequality --see also the definition of $C^\alpha(R)$-- that
$$|b_\alpha(X(t))|\leq C^\alpha(R) \qquad \text{for } t\geq 0.$$
Fix $\tau\in(0,T]$. Since $\phi$ is non negative, we have 
\begin{equation}\label{oo1}\begin{split}C^\alpha(R)+1\geq\int_0^\tau \frac{1}{2}X(t)^2+\frac{1}{2}(b_\alpha(X(t))-\dot X(t))^2dt\\ \geq \int_0^\tau \frac{1}{2}(b_\alpha(X(t))-\dot X(t))^2dt\,.\end{split} \end{equation}
Moreover, since for all $A>0$ and $y\in \RR$
$$\frac{1}{2}y^2\geq A |y|-A^2,$$
choosing $A=1/\sqrt{\tau}$ and $y=y(t)=b^\alpha(X(t))-\dot X(t)$, we get 
\begin{align*}
\int_0^\tau \frac{1}{2}(b^\alpha(X(t))-\dot X(t))^2dt&\geq \frac{1}{\sqrt{\tau}}\int_0^\tau |b_\alpha(X(t))-\dot X(t)|dt-1\\
&\geq \frac{1}{\sqrt{\tau}}\int_0^\tau |\dot X(t)|- |b_\alpha(X(t))|dt-1\\
&\geq \frac{1}{\sqrt{\tau}}\int_0^t |\dot X(t)|- C^\alpha(R)dt-1\\
&\geq\frac{1}{\sqrt{\tau}}|X(\tau)-x|- \sqrt{\tau}C^\alpha(R)-1
\end{align*}
and this, together with \eqref{oo1} implies \eqref{sigma}.

Now, by Lemma \ref{l42}, we readily get
$$u^\alpha(x,t)-\phi(x)\leq C^\alpha_3(R)t\leq \nu_R(t)\quad \text{for }(x,t)\in [0,R]\times\RR_+.$$
On the other hand, for $t\in \RR_+$ and $\varepsilon\in(0,1)$ there exists $X\in \mathcal C^+(x,t)$ such that
$$u^\alpha(x,t)+\varepsilon> \int_0^t \frac{1}{2}X^2(s)+\frac{1}{2}(b(X(s))-\dot X(s))^2)+\phi(x(t))\geq \phi(X(t)).$$
In view of the arguments above, $X$  verifies \eqref{sigma}. 
This, together with the arbitrariness of $\varepsilon$,  implies
$$u^\alpha(x,t)-\phi(x)\geq -\omega_{\phi,R}(\sigma_R(t))\geq -\nu_R(t) $$
and completes the proof. 
\end{proof}

\begin{lemma}\label{gammamodulus} Let $\alpha\in[0,1]$. For each $R>0$ there exists a modulus $\gamma_R$ for $u^\alpha(\cdot,T)$ in $[0,R]$, for all $T>0$. More precisely,  for each $x,y\in [0,R]$ and $T>0$
$$|u^\alpha(x,T)-u^\alpha(y,T)|\leq \gamma_R(|x-y|).$$
Moreover $\gamma_R$ depends on $\phi$ only via $\omega_{R,\phi}$ and $||\phi||_{L^\infty([0,R])}$.
\end{lemma}
\begin{proof}
We first let $T\leq |x-y|$. Taking $\nu_R$ as in  Lemma \ref{l456}, and by enlarging it if necessary, we may assume without loss of generality $|\phi(x)-\phi(y)|\leq \nu_R(|x-y|)$ so that
\begin{align*}
|u^\alpha(x,T)-u^\alpha(y,T)|&\leq |u^\alpha(x,T)-\phi(x)|+|\phi(x)-\phi(y)|+|\phi(y)-u^\alpha(y,T)|\\
&\leq 3\nu_R(|x-y|).\end{align*}
Assume now $T>|x-y|$ and, by swapping the variable's names $x$ and $y$ if necessary, assume that $u^\alpha(x,T)\leq u^\alpha(y,T)$. Fix $\varepsilon\in(0,1)$ and select $X\in\mathcal C^+(x,T)$ such that
$$u^\alpha(x,T)+\varepsilon\geq \int_0^T \frac{1}{2}X(t)^2+\frac{1}{2}(b_\alpha(X(t))-\dot X(t))^2dt+\phi(X(T)).$$ 
Since $\varepsilon<1$, we know by Lemma \ref{l29} that 
$$|X(t)|\leq C^\alpha_2\quad \text{for } t\in(0,T].$$
On the other hand,  replacing $\nu_R$ by $\nu_{C_2^\alpha(R)}$ if necessary, we have 
$$|u^\alpha(\hat x,|x-y|)-\phi(\hat x)|\leq \nu_R(|x-y|))\quad \text{for } \hat x\in [0,C_2^\alpha(R)].$$
Applying above inequality to $\hat x=X(T-|x-y|)$, we get
\begin{equation}\label{bu}\begin{split}u^\alpha(x,T)+\varepsilon&\geq \int_0^{T-|x-y|} \frac{1}{2}X(t)^2+\frac{1}{2}(b_\alpha(X(t))-\dot X(t))^2dt\\
&\quad +u^\alpha(X(T-|x-y|),|x-y|)\\
&\geq \int_0^{T-|x-y|}  \frac{1}{2}X(t)^2+ \frac{1}{2}(b_\alpha(X(t))-\dot X(t))^2dt\\
&\quad +\phi(X(T-|x-y|))-\nu_R(|x-y|).\end{split}
\end{equation}

Now, let $Y\in \mathcal C^+(y,T)$ be 
$$Y(t):=\begin{cases}
y+\frac{t}{|x-y|}(x-y)& \text{for } 0\leq t\leq |x-y|\\
X(t-|x-y|)& \text{for } |x-y|< t\leq T.
\end{cases}$$
Then, for $t\in[0,|x-y|]$ one has $Y(t)\in[0,R]$ and $\dot Y(t)\in[-1,1]$. Setting 
$$\tilde C^\alpha(R):=\frac{1}{2}\sup\{\xi^2+(b_\alpha(\xi)-\eta)^2\mid (\xi,\eta)\in [0,R]\times [-1,1]\}$$
one gets, also using the last inequality in \eqref{bu},
\begin{align*}
u^\alpha(y,T)-u^\alpha(x,T)\leq &\int_0^T \frac{1}{2}Y(t)^2+\frac{1}{2}(b_\alpha(Y(t))-\dot Y(t))^2dt + \phi(Y(T))+\varepsilon \\
&-\int_0^{T-|x-y|} \frac{1}{2}X(t)^2+\frac{1}{2}(b_\alpha(X(t))-\dot X(t))^2dt\\
& -\phi(X(T-|x-y|))+\nu_R(|x-y|)\\
=&\int_0^{|x-y|} \frac{1}{2}Y(t)^2+\frac{1}{2}(b_\alpha(Y(t))-\dot Y(t))^2dt +\nu_R(|x-y|)+\varepsilon\\
\leq &\tilde C^\alpha(R)|x-y|+\nu_R(|x-y|)+\varepsilon.\end{align*}
In particular,  $u^\alpha(y,T)\geq u^\alpha(x,T)$ implies
\begin{align*}
    |u^\alpha(y,T)-u^\alpha(x,T)|&=u^\alpha(y,T)-u^\alpha(x,T)\\&\leq \gamma_R(|x-y|):=\tilde C^\alpha(R)|x-y|+\nu_R(|x-y|)\,,
\end{align*}
and this concludes the proof of the local uniform continuity of $u^\alpha$ in $x$. Since $\tilde C^\alpha(R)$ is independent from $\phi$, we deduce from Lemma \ref{l456} and from the definition of $C_2^\alpha(R)$ and of $C^\alpha(R)$ that $\nu_R$, hence $\gamma_R$, depends on $\phi$ only via $\omega_{\phi,R}$ and $||\phi||_{L^\infty([0,R])}$.
\end{proof}

\begin{proposition}
[Dynamic Programming Principle]\label{DPP}
Let $\alpha\in[0,1]$. For all $x\geq 0$ and for all $S,T>0$
\begin{equation*}
    \begin{split}
        u^\alpha(x,T+S)=\inf\left\{\int_0^T \frac{1}{2}X^2(t)+\frac{1}{2}(b_\alpha(X(t))-\dot X(t))^2dt+u^\alpha(X(T),S)\right.\\
        \mid X\in \mathcal C^+(x,T)\Bigg\}.
    \end{split}
\end{equation*}
\end{proposition}
\begin{proof}
Set for brevity
\begin{equation*}
    \begin{split}
       \tilde u^\alpha(x,T,S):=\inf\left\{\int_0^T \frac{1}{2}X^2(t)+\frac{1}{2}(b_\alpha(X(t))-\dot X(t))^2dt+u^\alpha(X(T),S)\right.\\
        \mid X\in \mathcal C^+(x,T)\Bigg\}.
    \end{split}
\end{equation*}
Fix $T,S>0$, $x\geq0$ and $\varepsilon>0$. Let $X\in \mathcal C^+(x,T)$ be such that
$$\int_0^T \frac{1}{2}X^2(t)+\frac{1}{2}(b_\alpha(X(t))-\dot X(t))^2dt+u^\alpha(X(T),S)\leq \tilde u^\alpha(x,T,S)+\frac{\varepsilon}{2}.$$
Let $Y\in\mathcal C^+(X(T),S)$ be such that
$$\int_0^S \frac{1}{2}Y^2(t)+\frac{1}{2}(b_\alpha(Y(t))-\dot Y(t))^2dt+\phi(Y(S))< u^\alpha(X(T),S)+\frac{\varepsilon}{2}.$$
Then the map 
$$Z(t):=\begin{cases}
X(t)& \text{if } t\in[0,T]\\
Y(t-T)& \text{if } t\in[T,T+S]\end{cases}$$
belongs to $\mathcal C^+(x,T+S)$ and it satisies
\begin{align*}
    u^\alpha(x,T+S)&\leq \int_0^{T+S} \frac{1}{2}Z^2(t)+\frac{1}{2}(b_\alpha(Z(t))-\dot Z(t))^2dt+\phi(Z(S+T))\\
    &< \int_0^T \frac{1}{2}X^2(t)+\frac{1}{2}(b_\alpha(X(t))-\dot X(t))^2dt+u^\alpha(X(T),S)+\frac{\varepsilon}{2}\\
    &<\tilde u^\alpha(x,T,S)+\varepsilon.
\end{align*}
By the arbitrariness of $\varepsilon$, one deduces $u^\alpha(x,T+S)\leq \tilde u^\alpha(x,T,S)$.
Now, to prove the inverse inequality, let $\varepsilon>0$ and $X\in \mathcal C^+(x,T+S)$ be such that 
$$\int_0^{T+S} \frac{1}{2}X^2(t)+\frac{1}{2}(b_\alpha(X(t))-\dot X(t))^2dt+\phi(X(S+T))<u^\alpha(x,T+S)+\varepsilon.$$
It follows that
\begin{align*}
    \tilde u^\alpha(x,T,S)\leq &
    \int_0^{T} \frac{1}{2}X^2(t)+\frac{1}{2}(b_\alpha(X(t))-\dot X(t))^2dt+u^\alpha(X(T),S)\\
    &\int_0^{T} \frac{1}{2}X^2(t)+\frac{1}{2}(b_\alpha(X(t))-\dot X(t))^2dt\\
    &+\int_0^{S} \frac{1}{2}X^2(t+T)+\frac{1}{2}(b_\alpha(X(t+T))-\dot X(t+T))^2dt\\
    &+\phi(X(T+S))\\
    <& u^\alpha(x,T+S)+\varepsilon.
\end{align*}
Then $\tilde u^\alpha(x,T,S)\leq u^\alpha(x,T+S)$ and this concludes the proof.
\end{proof}

 We prolong the definition of $u^\alpha$ 
by setting $u^\alpha(x,0)=\phi(x)$ for $x\geq0$. 
\begin{proposition}\label{pc} $u^\alpha\in C([0,+\infty)\times [0,+\infty))$.
\end{proposition}
\begin{proof}
Let $R>0$. By Lemma \ref{gammamodulus} there exists a modulus of continuity $\gamma_R$ such that for every $S>0$, $x,y\in [0,R]$ one has 
$|u^\alpha(x,S)-u^\alpha(y,S)|\leq \gamma_R(|x-y|).$ In other words, setting $\bar \phi(\cdot):=u^\alpha(\cdot,S)$ we have that $\bar \phi$ is a  locally uniformly continuous map, and by Lemma \ref{l28} it attains its global minimum at $x=0$.
By Proposition \ref{DPP}, for every $T>0$
\begin{equation*}
    \begin{split}
        u^\alpha(x,T+S)=\inf\left\{\int_0^T \frac{1}{2}X(t)^2+\frac{1}{2}(b_\alpha(X(t))-\dot X(t))^2dt+\bar \phi(X(T))\right.\\\mid X\in \mathcal C^+(x,T)\Bigg\}.
    \end{split}
\end{equation*}
Applying Lemma \ref{l456} to $\bar u^\alpha(x,T):=u^\alpha(x,T+S)$ and to $\bar \phi$, we deduce that  there exists a modulus of continuity $\bar \nu_R$ satisfying
$$|u^\alpha(x,S+T)-u^\alpha(x,S)|=|\bar u^\alpha(x,T)-\bar \phi(x)|\leq \bar \nu_R(T).$$
Note that $\bar \nu_R$ depends on $\bar \phi$ (hence on $u^\alpha(\cdot,S)$) only  via $||\bar \phi||_{L^\infty([0,R])}$ and via the modulus of continuity of $\bar \phi$. In particular, by Lemma \ref{l28}, we have 
$||\bar \phi||_{L^\infty([0,R])}\leq C_1^\alpha(R)$.
On the other hand, the modulus of continuity of $\bar \phi$ is simply $\gamma_R$, i.e. the modulus of continuity of $u^\alpha(\cdot, S)$, which is independent from $S$ by Lemma \ref{gammamodulus}.   We then conclude that 
$$|u^\alpha(x,s)-u^\alpha(x,t)|\leq \bar \nu_R(|s-t|)\qquad \text{for }x\in B(0,R),\, t,s\in[0,+\infty)$$ 
and, consequently, we deduce the continuity of $u^\alpha(x,\cdot)$ for all $x\in [0,R]$. Since $R$ is arbitrary, the proof is complete. 
\end{proof}

\section{Asymptotic solutions}
In this section, we consider the infinite horizon problem of minimizing with respect to $\xi$
\begin{equation}\label{pbinf}\int_0^{+\infty} \frac{I^2(t)}{2}+\frac{\xi^2(t)}{2}dt
\qquad \text{subject to \eqref{controlya} and } I(t)\geq 0, t>0\end{equation}
and the associated value function 
$$v^\alpha(x):=\inf_{\xi\in \mathcal A_\infty}\int_0^{+\infty} \frac{I^2(t;\xi,x)}{2}+\frac{\xi^2(t)}{2}dt\,,$$
where $\mathcal A_\infty\subset L^1((0,+\infty))$ is the set of admissible controls, i.e.,  $\xi\in \mathcal A_\infty$ if $I(t;\xi,I_0)\geq 0$ for all $t>0$.
Our aim is to prove that, for a suitable notion of convergence, the value function $u^\alpha(x,T)$ of the finite horizon problem in Section \ref{s3}   tends to $v^\alpha(x)$ as $T\to+\infty$.
We remark the lack of a discount factor in the optimization problem \eqref{pbinf}, so our first goal is to prove that $v^\alpha(x)$ is finite for all $x\geq 0$. To this end, we introduce a representation formula for $v^\alpha$, based on the functions $d^\alpha:[0,+\infty)\times [0,+\infty)\to [0,+\infty)$
and $\psi^\alpha:[0,+\infty)\to [0,+\infty)$ respectively defined by 
\begin{equation*}
    \begin{split}
        d^\alpha(x,y):=\inf\left\{\int_0^T \frac{1}{2}X(t)^2+\frac{1}{2}(b_\alpha(X(t))-\dot X(t))^2dt\right.\\\mid T>0,\, X\in \mathcal C^+(x,y,T)\Bigg\}
    \end{split}
\end{equation*}
and
\begin{equation*}
    \begin{split}
  \psi^\alpha(x):=\inf\left\{\int_0^T \frac{1}{2}X(t)^2+\frac{1}{2}(b_\alpha(X(t))-\dot X(t))^2dt+\phi(X(T))\right.\\\mid T>0,\, X\in \mathcal C^+(x,T)\Bigg\}\,,      
    \end{split}
\end{equation*}
where, for $x,y\geq0$, $\mathcal C^+(x,y,T)$ denotes the space of non-negative, absolutely continuous functions  $X: [0,T]\to[0,+\infty)$  satisfying  $(X(0),X(T))=(x,y)$.
\begin{theorem}\label{thmva}
For all $\alpha\in[0,1]$ and $x\geq 0$
\begin{equation}\label{vchar}
v^\alpha(x)=d^\alpha(x,0)+\psi^\alpha(0).
\end{equation}
In particular 
\begin{equation}\label{vcharbond}0\leq v^\alpha(x)\leq(x+||b_\alpha||_{L^\infty([0,x])})^2+\psi^\alpha(0)\end{equation}
for all $x\geq 0$.
Moreover $v^\alpha$ is a viscosity solution of 
\begin{equation}\label{hjstat}\begin{cases}
-b_\alpha(x)Dv+\frac{1}{2}Dv^2-\frac{1}{2}x^2=0, \quad x\in\RR_+\\
v(0)=\phi(0).
\end{cases}
\end{equation}
\end{theorem}
 The proof is postponed in Section \ref{s41} below. We finally state our main result, whose proof is showed in Section \ref{s42}. 
\begin{theorem}\label{thm1}
For all $\alpha\in[0,1]$, the value function $u^\alpha$ defined in \eqref{value} satisfies for all $R>0$ 
\begin{equation}\label{limit}
\lim_{T\to+\infty}\max_{x\in[ 0,R]}|u^\alpha(x,T)-v^\alpha(x)|=0.
\end{equation}
\end{theorem}
\begin{remark}\label{rmk1}
An explicit viscosity solution of \eqref{hjstat} is provided by 
$$\bar v^\alpha(x):=\phi(0)+\int_0^x b_\alpha(s)+\sqrt{b_\alpha^2(s)+s^2}ds\,.$$
 In particular, $\bar v_\alpha$ is smooth, positive, increasing and it satisfies $\bar v(x)\to+\infty$ as $x\to +\infty$.
 Moreover, for $\alpha=1$,  using the Wolfram Mathematica software, we obtain the following closed form for $\bar v^1$:
  \begin{align*}
  \bar v^1(x)=&\phi(0)+C_0(\beta,\gamma,N)\\
  &-\frac{\beta  x^3}{3 N}+\frac{1}{2} x^2 (\beta -\gamma )+\frac{N^2}{\beta^2}\left(\frac{1}{3}
    \left(y^2(x)+1\right)^{3/2}\right.\\&\left.-\frac{1}{2} (\beta -\gamma)y(x)
   \sqrt{y^2(x)+1} -\frac{1}{2} (\beta -\gamma )
   \sinh ^{-1}\left(y(x)\right)\right)\,,
   \end{align*}
   where $y(x):=\beta-\gamma-\frac{\beta}{N}x$ and
   \begin{equation*}
       \begin{split}
       C_0(\beta,\gamma,N)=\frac{N^2}{\beta^2}\left(\frac{1}{2}((\beta -\gamma )^2+1)^{1/2} (\beta -\gamma )^2-\frac{1}{3} \left((\beta -\gamma )^2+1\right)^{3/2}\right.\\\left.+\frac{1}{2} (\beta -\gamma ) \sinh ^{-1}(\beta -\gamma )\right).
   \end{split}
   \end{equation*}
\end{remark}

\subsection{Preliminary results and proof of Theorem \ref{thmva}}\label{s41}
Our first result proves the first part of the claim of Theorem \ref{thmva}, that is  \eqref{vchar} and \eqref{vcharbond} 
\begin{lemma}
For all $\alpha\in[0,1]$ and $x\geq 0$
\begin{equation*}
v^\alpha(x)=d^\alpha(x,0)+\psi^\alpha(0).
\end{equation*}
In particular 
$$0\leq v^\alpha(x)\leq(x+||b_\alpha||_{L^\infty([0,x])})^2+\psi^\alpha(0)$$
for all $x\geq 0$.
\end{lemma}
\begin{proof}

Fix $\varepsilon>0$ and let $T>0$ and $X\in C^+(x,0,T)$ be such that
$$\int_0^T \frac{1}{2}X(t)^2+\frac{1}{2}(b_\alpha(X(t))-\dot X(t))^2dt<d_\alpha(x,0)+\varepsilon.$$
Then prolonging the definition of $X$ to $(T,+\infty)$ by setting $X(t)=0$ for all $t>T$, and letting $\xi\in L^1((0,+\infty))$ be such that $\xi(t)=\dot X(t)-b_\alpha(X(t))$ a.e., one deduces that $X(t)=I(t;\xi,x)$ for some admissible trajectory, and consequently
$v^\alpha(x)< d^\alpha(x,0)+\varepsilon.$
By the arbitrariness of $\varepsilon$ one can deduce $v^\alpha(x)\leq d^\alpha(x,0)$ and, since $\psi^\alpha(0)$ is nonnegative, also $v^\alpha(x)\leq d^\alpha(x,0)+\psi^\alpha(0)$. The inverse inequality can be proved similarly. 

In view of \eqref{vchar}, setting $X\in C^+(x,0,1)$ given by $X(t):=x(1-t)$ one has $X(t)\in[0,x]$ for all $t\in[0,1]$ and
\begin{align*}
    v^\alpha(x)&=d^\alpha(x,0)+\phi^\alpha(0)\\
    &\leq \int_0^1 \frac{1}{2}(x(1-t))^2+\frac{1}{2}(b_\alpha(x(1-t))-x)^2dt\\
    &\leq (x+||b_\alpha||_{L^\infty([0,x])})^2+\phi^\alpha(0).
\end{align*}
\end{proof}

\begin{lemma}\label{l1}
For all $\alpha\in[0,1]$ the function $d^\alpha$ is a locally Lipschitz continuous  function in $[0,+\infty)\times [0,+\infty)$. In particular, $v^\alpha$ is locally Lipschitz continuous in $[0,+\infty)$.

Moreover, the following properties hold:
$$d^\alpha(x,y)\geq0;\quad d^\alpha(x,y)\leq d^\alpha(x,z)+d^\alpha(z,y);\quad d(x,x)=0\,\, \forall x,y,z\in [0,+\infty).$$

\end{lemma}
\begin{proof}
\begin{enumerate}
\item[1.] Fix $x,y\in[0,+\infty)$ and $\varepsilon>0$. Choose $T>0$ and $X\in \mathcal C^+(x,y,T)$ so that 
$$d^\alpha(x,y)+\varepsilon>\int_0^T\frac{1}{2} X(t)^2+\frac{1}{2}(b_\alpha(X(t))-\dot X(t))^2dt.$$
Since the integrand in above expression is non-negative, we deduce by the arbitrariness of $\varepsilon$ that $d^\alpha(x,y)\geq 0$.
\item[2.] Fix $R>0$ and $x,y\in [0,R]$. We set as in the proof of Lemma \ref{gammamodulus}
\begin{align*}
\tilde C^\alpha_R:=&\max\left\{ \frac{1}{2}x^2+\frac{1}{2}(b_\alpha(x)-\xi)^2\mid (x,\xi)\in[0,R]\times [-1,1]\right\}\,.
\end{align*}
We first assume that $x\not=y$, we define the curve $X\in \mathcal C^+(x,y,|x-y|)$ by 
$$X(t)=x-\frac{ t}{|x-y|}(x-y)\qquad \text{for } 0\leq t\leq |x-y|,$$
and we observe that $X(t)\in[0, R]$ and $\dot X(t)\in [-1,1]$ for all $t\in[0,|x-y|]$. Hence
\begin{align*}
d^\alpha(x,y)&\leq \int_{0}^{|x-y|}\frac{1}{2}X(t)^2+\frac{1}{2}(b_\alpha(X(t))-\dot X(t))^2 dt\\
&\leq \tilde C^\alpha_R |x-y|.
\end{align*}
Now, we consider the case $x=y$. Fix any $T>0$ and set $X(t)=x$ for $t\in[0,T]$. Then we have
\begin{align*}
d^\alpha(x,y)&\leq \int_{0}^{T}\frac{1}{2}x^2+\frac{1}{2}b^2_\alpha(x) dt\leq \tilde C^\alpha_R T.
\end{align*}
Therefore 
\begin{equation}\label{dcont}
d^\alpha(x,y)\leq \tilde C^\alpha_R|x-y|\quad \forall x,y\in[ 0,R].
\end{equation}
In particular $d(0,0)=0$.
\item[3.] Let $x,y,z\in[0,+\infty)$. Let $T,S>0$, $X\in \mathcal C^+(x,z,T)$ and $Y\in \mathcal C^+(z,y,S)$. Define $Z\in \mathcal C^+(x,y,T+S)$ by
$$
Z(t)=\begin{cases}
X(t) \quad &\text{ for } 0\leq t\leq T,\\
Y(t-T)\quad &\text{ for } T< t\leq T+S.
\end{cases}
$$
We have
\begin{equation*}\begin{split}
d^\alpha(x,y)\leq& \int_{0}^{T+S}\left(\frac{1}{2}Z(t)^2+\frac{1}{2}(b_\alpha(Z(t))-\dot Z(t))^2\right) dt\\
=&\int_{0}^{T}\frac{1}{2}X(t)^2+\frac{1}{2}(b_\alpha(X(t))-\dot X(t))^2 dt\\
&+\int_{0}^{S}\frac{1}{2}Y(t)^2+\frac{1}{2}(b_\alpha(Y(t))-\dot Y(t))^2dt\,.
\end{split}\end{equation*}
By the arbitrariness of $X$ and $Y$, we deduce
\begin{equation}\label{triangular}
d^\alpha(x,y)\leq d^\alpha(x,z)+d^\alpha(z,y).\end{equation}
\item[4.] Finally, using \eqref{dcont} and \eqref{triangular}, we deduce that for all $R>0$ and for all $x,y,\xi,\eta\in[0,R]$
$$|d^\alpha(x,y)-d^\alpha(\xi,\eta)|\leq \tilde C^\alpha_R(|x-\xi|+|y-\eta|)\leq 2C^\alpha_R||(x,y)-(\xi,\eta)||.$$
The local Lipschitz continuity of $v^\alpha$ readily follows by the identity $v^\alpha(x)=d^\alpha(x,0)+\psi^\alpha(0)$.
\end{enumerate}
\end{proof}

\subsubsection{Proof of Theorem \ref{thmva}} It is left to prove that for all $\alpha\in[0,1]$, 
$v^\alpha$ is a viscosity solution of \eqref{hjstat}. 
\begin{proof}
Let $\varphi\in C^1((0,+\infty))$ and $\hat x\in (0,+\infty)$.  We first assume that $v^\alpha-\varphi$ attains a local maximum at $\hat x$. We then may assume, without loss of generality, that $v^\alpha(\hat x)-\varphi(\hat x)= 0$, so that $v^\alpha\leq \varphi$ in $B_\delta(\hat x)$ for some $\delta>0$. 
Setting $X(t)=\hat x+ (b_\alpha(\hat x)-D\varphi(\hat x))t$, by Lemma \ref{l1} and by the definition of $d^\alpha$, it follows that, for all $\varepsilon>0$ such that $X(\varepsilon)\in B_\delta(\hat x)$, 
\begin{align*}\varphi (\hat x)&=v^\alpha(\hat x)=d^\alpha(\hat x,0)+\psi^\alpha(0)\\
&\leq d^\alpha(\hat x,X(\varepsilon))+d^\alpha(X(\varepsilon),0)+\psi^\alpha(0)\\
&= d^\alpha(\hat x,X(\varepsilon))+v^\alpha(X(\varepsilon))\\
&\leq d^\alpha(\hat x,X(\varepsilon))+\varphi(X(\varepsilon)) \\
&\leq \frac{1}{2}\int_0^\varepsilon (X(t))^2+(b_\alpha(X(t))-\dot X(t))^2 dt+ \varphi(X(\varepsilon)) \,.
\end{align*}
Since $X(\varepsilon)=\hat x+(b_\alpha(\hat x)-D\varphi(\hat x))\varepsilon$, then for all sufficiently small $\varepsilon>0$,
$$
\frac{\varphi(\hat x)-\varphi(\hat x+((b_\alpha(\hat x)-D\varphi(\hat x))\varepsilon)}{\varepsilon}-\frac{1}{\varepsilon}\int_0^\varepsilon \frac{1}{2}X(t)^2+\frac{1}{2}(b_\alpha(X(t))-\dot X(t))^2 dt\leq0\,.
$$
Letting $\varepsilon\to 0^+$ and remarking that $X(0)=\hat x$ and $\dot X(0)=b_\alpha(\hat x)-D\varphi(\hat x)$, we obtain after a few computation
$$\frac{1}{2}D\varphi(\hat x)^2-\frac{1}{2}\hat x^2-b_\alpha(\hat x)D\varphi(\hat x)\leq 0.$$
We deduce, by the arbitrariness of $\varphi$, that $v^\alpha$ is a viscosity subsolution of \eqref{hjstat}.

Now we assume that $v^\alpha-\varphi$ attains a local minimum at $\hat x$. Again, we may assume, without loss of generality, that $v^\alpha(\hat x)-\varphi(\hat x)= 0$, so that $v^\alpha\geq \varphi$ in $B_\delta(\hat x)$ for some $\delta>0$. 
Let $C^\alpha(\hat x)$ be like in Lemma \ref{l456} (with $R=\hat x$), $\varepsilon\in(0,\min\{1,\delta^2 /(2(C^\alpha(\hat x)+1))^2\})$ and choose $T>0$ and $X_\varepsilon \in \mathcal{ C}^+(\hat x, 0,T)$ such that
$$d^\alpha(\hat x ,0)+\varepsilon^2> \int_0^T \frac{1}{2} X_\varepsilon(t)^2+\frac{1}{2}(b_\alpha(X_\varepsilon(t))-\dot X_\varepsilon(t))^2 dt\,.$$
 If $T<1$, we may prolong the definition of $X_\varepsilon$ to $[0,1]$ by setting $X_\varepsilon\equiv 0$ in $(T,1]$, and obtaining the above inequality to hold for $T=1$, as well. Therefore we assume, without loss of generality, that $T\geq1$.
Arguing as in the proof of Lemma \ref{l456}, one can deduce that 
$$|X_\varepsilon(t)-\hat x|\leq tC^\alpha(\hat x)+\sqrt{t}(C^\alpha(\hat x)+2)\quad \text{for } t\in[0,T].$$
In particular,  we have $\varepsilon<1\leq T$ and 
$$|X_\varepsilon(t)-\hat x|\leq  t C^\alpha(\hat x)+\sqrt{t}(C^\alpha(\hat x)+2)<2\sqrt{\varepsilon}(C^\alpha(\hat x)+1)\leq \delta  \qquad \forall t\in[0,\varepsilon].$$
As soon as $t\in[0,\varepsilon]$ we can deduce $v^\alpha(X_\varepsilon(t))\geq \varphi(X_\varepsilon(t)))$. Then

\begin{align*}\varphi (\hat x)+\varepsilon^2=&v^\alpha(\hat x)+\varepsilon^2=d^\alpha(\hat x,0)+\varepsilon^2+\psi^\alpha(0)\\
>& \int_0^T \frac{1}{2}X_\varepsilon(t)^2+\frac{1}{2}(b_\alpha(X_\varepsilon(t))-\dot X_\varepsilon(t))^2 dt+\psi^\alpha(0)\\
=&\int_0^{\varepsilon} \frac{1}{2}X_\varepsilon(t)^2+\frac{1}{2}(b_\alpha(X_\varepsilon(t))-\dot X_\varepsilon(t))^2 dt\\+
&\int_0^{T-\varepsilon} \frac{1}{2}X_\varepsilon(t+\varepsilon)^2+\frac{1}{2}(b_\alpha(X_\varepsilon(t+\varepsilon))-\dot X_\varepsilon(t+\varepsilon))^2 dt+\psi^\alpha(0)\\
\geq&\int_0^{\varepsilon} \frac{1}{2}X_\varepsilon(t)^2+\frac{1}{2}(b_\alpha(X_\varepsilon(t))-\dot X_\varepsilon(t))^2 dt+v^\alpha(X_\varepsilon(\varepsilon))\\
\geq&\int_0^{\varepsilon} \frac{1}{2}X_\varepsilon(t)^2+\frac{1}{2}(b_\alpha(X_\varepsilon(t))-\dot X_\varepsilon(t))^2 dt+\varphi(X_\varepsilon(\varepsilon)). 
\end{align*}
Now,
$$\varphi((X_\varepsilon(\varepsilon))-\varphi(\hat x)=\varphi((X_\varepsilon(\varepsilon))- \varphi(X_\varepsilon(0))=\int_0^\varepsilon D\varphi(X_\varepsilon(t))\dot X_\varepsilon(t)dt\,,$$
hence
\begin{equation}\label{ah}
\varepsilon^2>\int_0^{\varepsilon} \frac{1}{2}X_\varepsilon(t)^2+\frac{1}{2}(b_\alpha(X_\varepsilon(t))-\dot X_\varepsilon(t))^2+ D\varphi(X_\varepsilon(t))\dot X_\varepsilon(t)dt\,.
\end{equation}
Since  $y^2/2\geq z y -z^2/2 $ for all $y,z\in\RR$,  setting $y=y(t)=b_\alpha(X_\varepsilon(t))-\dot X_\varepsilon(t)$ and $z=z(t)=D\varphi(X_\varepsilon(t))$, we get
$$\frac{1}{2}(b_\alpha(X_\varepsilon(t))- \dot X_\varepsilon(t))^2\geq (b_\alpha(X_\varepsilon(t)-\dot X_\varepsilon(t)))D\varphi(X_\varepsilon(t))-\frac{1}{2}D\varphi(X_\varepsilon(t))^2\quad \forall t\in[0,\varepsilon]\,.$$
This, together with \eqref{ah}, implies 
\begin{align*}\varepsilon&> \frac1\varepsilon\int_0^{\varepsilon} \frac{1}{2}(X_\varepsilon(t))^2+b_\alpha(X_\varepsilon(t))D\varphi(X_\varepsilon(t)) -\frac{1}{2}D\varphi(X_\varepsilon(t))^2dt\,.
\end{align*}
Letting $\varepsilon\to 0^+$, we finally conclude
$$\frac{1}{2}D\varphi(\hat x)^2-\frac{1}{2}\hat x^2-b_\alpha(\hat x)D\varphi(\hat x)\geq 0.$$
We deduce, by the arbitrariness of $\varphi$, that $v^\alpha$ is a also viscosity supersolution of \eqref{hjstat} and this concludes the proof.
\end{proof}

\subsection{Preliminary results and proof of Theorem \ref{thm1}}\label{s42}
We begin with the following  regularity result for $\psi$.
\begin{lemma}For all $\alpha\in[0,1]$, 
$\psi^\alpha$ is locally Lipschitz continuous.
\end{lemma}
\begin{proof}
 Fix $R>0$ and choose $x,y\in [0,R]$ such that $\psi^\alpha(x)<\psi^\alpha(y)$. Recall from Lemma \ref{gammamodulus} the definition 
\begin{align*}
\tilde C^\alpha_R:=&\sup\left\{ \frac{1}{2}x^2+\frac{1}{2}(b_\alpha(x)-\xi)^2\mid (x,\xi)\in [0,R]\times [-1,1]\right\}.
\end{align*}
 Choose $T>0$ and $X\in \mathcal C^+(x,T)$ such that 
$$\psi^\alpha(x)+\varepsilon>\int_0^T\frac{1}{2} X(t)^2+\frac{1}{2}(b_\alpha(X(t))-\dot X(t))^2dt+\phi(X(T)).$$
Define $Y\in \mathcal C^+(y, T+|x-y|)$ by 
$$Y(t)=\begin{cases}
y+t\dfrac{x-y}{|x-y|}\quad&\text{ for } 0\leq t< |x-y|,\\
X(t-|x-y|)&\text{ for } |x-y|\leq T+|x-y|.
\end{cases}$$
We have
\begin{align*}
\psi^\alpha(y)\leq&\int_0^{T+|x-y|}\frac{1}{2}Y(t)^2+\frac{1}{2}(b_\alpha(Y(t))-\dot Y(t))^2dt+\phi(Y(T))\\
=&\int_0^{|x-y|}\frac{1}{2}Y(t)^2+\frac{1}{2}(b_\alpha(Y(t))-\dot Y(t))^2dt\\
&+\int_0^{T} \frac{1}{2}X(t)^2+\frac{1}{2}(b_\alpha(X(t))-\dot X(t))^2dt+\phi(X(T))\\
<&\tilde C^\alpha_R|x-y|+\varepsilon+\psi^\alpha(x).
\end{align*}
From this we deduce $|\psi^\alpha(y)-\psi^\alpha(x)|\leq C^\alpha_R|x-y|$ and this completes the proof.
\end{proof}
\begin{lemma}\label{l25}Let $\alpha\in[0,1]$, $R>0$, and $\varepsilon>0$. There exists a constant $T>0$ such that, for each $x,y \in [0,R]$, there exists $S\in(0,T]$ and $X\in \mathcal C^+(x,y,S)$ such that
$$d^\alpha(x,y)+\varepsilon>\int_0^S\frac{1}{2}X(t)^2+\frac{1}{2}(b_\alpha(X(t))-\dot X(t))^2dt.$$
\end{lemma}
\begin{proof}
Let $R>0$ and $\varepsilon\in(0,1)$. Fix $\bar x,\bar y\in [0,R]$ and choose $\bar T>0$ and $Y\in \mathcal C^+(\bar x,\bar y, \bar T)$ so that 
$$d(\bar x,\bar y)+\frac{\varepsilon}{4}>\int_0^{\bar T} \frac{1}{2}Y(t)^2+\frac{1}{2}(b_\alpha(Y(t))-\dot Y(t))^2dt.$$
Let
$$C_R:=\max\left\{\frac{1}{2}x^2+\frac{1}{2}(b_\alpha(x)+y)^2\mid |x|\leq 1+R, |y|\leq 1\right\}.$$
Fix $\delta\in(0,1)$ so that
$$2 C_R\delta\leq \frac{\varepsilon}{4}$$
and
$$|d^\alpha(x,y)-d^\alpha(\bar x,\bar y)|<\frac{\varepsilon}{2}\quad \text{for } x\in B(\bar x,\delta),\,y\in B(\bar y,\delta).$$
Let $x\in B(\bar x,\delta)\cap [0,+\infty)$ and $y\in B(\bar y,\delta)\cap [0,+\infty)$,  Define $\xi \in \mathcal C^+(x,\bar x,\delta)$ and $\eta \in \mathcal C^+(\bar y,y,\delta)$, respectively,  by
$$ \xi(t)=x+\frac{t}{\delta}(\bar x-x) \quad \text{for } 0\leq t\leq \delta$$
$$ \eta(t)=\bar y+\frac{t}{\delta}(y-\bar y) \quad \text{for } 0\leq t\leq \delta.$$
Noting that $\xi(t),\eta(t)\in [0,R+1]$ and $\dot \xi(t),\dot \eta(t) \in B (0,1)$ for all $t\in [0,\delta]$, we have
$$\int_0^\delta  \frac{1}{2}\xi(t)^2+\frac{1}{2}(b_\alpha(\xi(t))-\dot \xi(t))^2dt\leq C_R\delta$$
$$\int_0^\delta  \frac{1}{2}\eta(t)^2+\frac{1}{2}(b_\alpha(\eta(t))-\dot \eta(t))^2dt\leq C_R\delta.$$
Define the function $X\in \mathcal C^+(x,y,\bar T+2\delta)$ by 
$$X(t)=\begin{cases}\xi(t)& \text{ for } t\in [0,\delta],\\
Y(t-\delta) &  \text{ for } t\in [\delta,\bar T+\delta],\\
\eta(t-\bar T-\delta) &\text{ for } t\in [\bar T+\delta,\bar T+2\delta].
\end{cases}$$
Then we have
\begin{align*}
\int_0^{\bar T+2\delta} \frac{1}{2}X(t)^2&+\frac{1}{2}(b_\alpha(X(t))-\dot X(t))^2dt=\int_0^\delta  \frac{1}{2}\xi(t)^2+\frac{1}{2}(b_\alpha(\xi(t))-\dot \xi(t))^2dt\\
&+ \int_0^{\bar T} \frac{1}{2}Y(t)^2+\frac{1}{2}(b_\alpha(Y(t))-\dot Y(t))^2dt\\
&+ \int_0^{\delta} \frac{1}{2}\eta(t)^2+\frac{1}{2}(b_\alpha(\eta(t))-\dot \eta(t))^2dt\\
\leq&2C_R\delta+d(\bar x,\bar y)+\frac{\varepsilon}{4}\leq d^\alpha(\bar x,\bar y)+\frac{\varepsilon}{2}\\
<&d^\alpha(x,y)+\varepsilon.
\end{align*}
We deduce that for each $(\bar x,\bar y)\in[0,R]\times [0,R]$ there exist constants $\bar S>0$ and $\delta>0$ such that, for any $x\in B(\bar x,\delta)\cap [0,+\infty)$ and $y\in B(\bar y,\delta)\cap[0,+\infty)$, we have
$$d^\alpha(x,y)+\varepsilon>\int_0^{\bar S}\frac{1}{2}X(t)^2+\frac{1}{2}(b_\alpha(X(t))-\dot X(t))^2dt$$
for some $X\in \mathcal C^+(x,y,\bar S)$. From now on the proof goes as \cite[Lemma 2.5]{FIL06} (with small adaptations): we report it for completeness. By the compactness of $[0,R]\times[0,R]$, there exists a finite collection of $(x_k,y_k)\in [0,R]\times[0,R]$,
$S_k,\delta_k>0$, for $k=1,\dots,K$, such that
$$[0,R]\times[0,R]\subseteq \bigcup_{k=1}^K B(x_k,\delta_k)\times B(y_k,\delta_k)$$
and such that for any $k \in \{1,2\dots,K\}$,
$x \in B(x_k, \delta_k)\cap[0,+\infty)$, and $y \in B(y_k, \delta_k)\cap[0,+\infty)$,
$$d^\alpha(x,y)+\varepsilon>\int_0^{S_k}\frac{1}{2}X(t)^2+\frac{1}{2}(b_\alpha(X(t))-\dot X(t))^2dt$$
for some $X\in \mathcal C^+(x,y, S_k)$. Setting $T = \max_{1\leq k \leq K} S_k$, we observe that for any $(x, y) \in
[0,R] \times [0,R]$,
$$d(x, y) + \varepsilon >\int_0^{S_k}\frac{1}{2}X(t)^2+\frac{1}{2}(b_\alpha(X(t))-\dot X(t))^2dt$$
for some $S \in  (0,T]$ and $X \in C^+(x,y,S)$, and this concludes the proof.
\end{proof}

\begin{lemma}\label{l26}Let $\alpha\in[0,1]$, $R>0$, and $\varepsilon>0$. There exists a constant $T>0$ such that, for each $x\in [0,R]$, there exist $S\in(0,T]$ and $X\in \mathcal C^+(x,S)$ such that
$$\psi^\alpha(0)+\varepsilon>\int_0^S\frac{1}{2}X(t)^2+\frac{1}{2}(b_\alpha(X(t))-\dot X(t))^2dt+\phi(X(S)).$$
\end{lemma}

\begin{proof}The proof is similar to Lemma \ref{l25}, hence we omit it. 
\end{proof}

\subsubsection{Proof of Theorem \ref{thm1}} 
Fix $\alpha\in[0,1]$, $R>0$ and $\varepsilon>0$. Using Lemma \ref{l25} and Lemma \ref{l26}, let $T_{R,\varepsilon}$ be such that, for all $x\in [0,R]$, there exist $S,S'\in(0,T]$, $X\in \mathcal C^+(x,0,S)$ and $Y\in \mathcal C^+(0,S')$ satisfying 
$$d^\alpha(x,0)+\varepsilon\geq\int_0^S\frac{1}{2}X(t)^2+\frac{1}{2}(b_\alpha(X(t))-\dot X(t))^2dt$$
and 
$$\psi^\alpha(0)+\varepsilon\geq\int_0^{S'}\frac{1}{2}Y^2(t)+\frac{1}{2}(b_\alpha(Y(t))-\dot Y(t))^2dt+\phi(Y(S')).$$
Fix $t\geq 2T_{R,\varepsilon}$ and define 
$$Z(s):=\begin{cases}
X(s) & \text{for } s\in [0,S]\\
0 & \text{for } s\in (S,t-S']\\
Y(s-t+S')& \text{for } s\in (t-S',t].
\end{cases}$$
By Proposition \ref{DPP}, we get
\begin{align*}
u^\alpha(x,t)&\leq \int_0^{t}\frac{1}{2}Z^2(s)+\frac{1}{2}(b_\alpha(Z(s))-\dot Z(s))^2dt+\phi(Z(t))\\
&=\int_0^{S}\frac{1}{2}X^2(s)+\frac{1}{2}(b_\alpha(X(s))-\dot X(s))^2dt\\
&+
\int_0^{S'}\frac{1}{2}Y^2(s)+\frac{1}{2}(b_\alpha(Y(s))-\dot Y(s))^2dt+\phi(Y(S'))\\
&\leq d^\alpha(x,0)+\varepsilon+\psi^\alpha(0)+\varepsilon=v^\alpha(x)+2\varepsilon.
\end{align*}
Therefore, 
\begin{equation}\label{upper}
u^\alpha(x,t)\leq v^\alpha(x)+\varepsilon \quad \text{for } x\in [0,R],\,t\geq 2T_{R,\varepsilon}.
\end{equation}
 Now, choose $\varepsilon\in(0,1)$. We prove that there exists a sufficiently large $T$ such that
\begin{equation}\label{lower}
v^\alpha(x)\leq u^\alpha(x,t)+\varepsilon \quad \text{for } x\in[0,R],\,t\geq T.
\end{equation}

Recall the definition \eqref{Bdef} and consider the integral function
$$\hat B_\alpha(y,x):=\hat B_\alpha(y)-\hat B_\alpha(x)=\int_x^y b_\alpha(z)dz.$$
Recall the definition of endemic population $E:=N(1-1/\rho)$ and note that, for all $x\in [0,R]$,
$$\hat B_\alpha(y,x)\leq \hat B_\alpha(E)-\hat B_\alpha(x)\leq \hat B_\alpha(E)-\min\{\hat B_\alpha(x) \mid x\in [0,R]\}=: C_4^\alpha(R)\,.$$
Let $C^\alpha_1(R)$ and $C^\alpha_2(R)$ be respectively like in \eqref{C1} and \eqref{C2}. By taking a larger $C^\alpha_1(R)$ if necessary, also assume 
\begin{equation}\label{c1}C_1^\alpha(R)\geq \max\{\frac{1}{2}x^2+\frac{1}{2}(b_\alpha(x)+y)^2\mid |x|\leq C_2(R),\, |y|\leq 1\}.\end{equation}
Set
$$\bar C^\alpha(R):=C_1^\alpha(R)+C_4^\alpha(R).$$
Fix $\delta\in(0,1)$ so that 
$$2\bar C^\alpha(R)\delta <\varepsilon$$
and define $\bar \gamma=\min\{\frac{1}{2}z^2+\frac{1}{2}b_\alpha^2(z)\mid |z|\geq \delta\}$.
Then define
$$T:=\frac{1}{\bar\gamma}(2\bar C^\alpha(R)+1).$$
 Let $t\geq T$ and $ X\in \mathcal C^+(x,T)$ such that
$$u^\alpha(x,t)+\varepsilon>\int_0^{t}\frac{1}{2} X^2(s)+\frac{1}{2}(b_\alpha(X(s))-\dot X(s))^2dt+\phi(X(t))\,.$$
Then
\begin{align*}
u^\alpha(x,t)+1&>\int_0^{t}\frac{1}{2} X^2(s)+\frac{1}{2}b^2_\alpha(X(s))ds -\int_0^t b_\alpha(X(s))\dot X(s)ds\\
&=\int_0^{t}\frac{1}{2} X^2(s)+\frac{1}{2}b_\alpha(Y(s))^2ds -\int_x^{X(t)}b_\alpha(z)dz\\
&=\int_0^{t}\frac{1}{2} X^2(s)+\frac{1}{2}b_\alpha(Y(s))^2ds -B_\alpha(X(t),x)\\
&\geq \int_0^{t}\frac{1}{2} X^2(s)+\frac{1}{2}b_\alpha(Y(s))^2ds - C_4^\alpha(R). 
\end{align*}
In view of Lemma \ref{l28}, we get
\begin{equation}\label{contr}
\int_0^{t}\frac{1}{2} X^2(s)+\frac{1}{2}b_\alpha^2(X(s))ds< C_1(R)+C_4^\alpha(R)+1=\bar C^\alpha(R)+1.
\end{equation}
 We now prove that 
 \begin{equation}\label{minumy}
 X(s)\leq\delta\quad \text{for some $s\in [0,t]$.}\end{equation} Assume on the contrary that $X(s)>\delta$ for all $s\in[0,t]$. Then, by the definition of $\bar\gamma$,
\begin{equation*}
\int_0^{t}\frac{1}{2} X^2(s)+\frac{1}{2}b_\alpha(X(s))^2ds\geq \bar\gamma t\geq \bar \gamma T=2\bar C_1(R)+1\,,
\end{equation*}
  and we get the required contradiction with \eqref{contr}. Let $\tau \in [0,t]$ be such that $0\leq X(\tau)\leq \delta$ and note that, by Lemma \ref{l29},
  $$|X(s)|\leq C^\alpha_2(R)\quad \text{for } s\in [0,t].$$

 Now, define $\xi\in \mathcal C^+(Y(\tau),0,\delta)$ and $\eta \in \mathcal C(0,Y(\tau),\delta)$ respectively as 
 $$Y(s):=X(s)-\frac{s}{\delta}X(s)\quad \text{for } s\in[0,\delta]\,,$$
 $$Z(s):=\frac{s}{\delta}X(s)\quad \text{for } s\in[0,\delta]\,.$$
 Noting that 
 $$Y(s),Z(s) \in[0,C_2(R)],\quad \dot Y(s),\dot Z(s)\in B(0,R)\qquad \text{for } s\in [0,t],$$
in view of \eqref{c1}, we have
 $$\int_0^\delta \frac{1}{2}Y^2(s)+\frac{1}{2}(b_\alpha(Y(s))-\dot Y(s))^2ds \leq C^\alpha_1(R)\delta$$
 and 
 $$\int_0^\delta \frac{1}{2}Z^2(s)+\frac{1}{2}(b_\alpha(Z(s))-\dot Z(s))^2ds \leq C^\alpha_1(R)\delta.$$
 Define the function $\eta\in\mathcal C(x,\delta, t+2\delta)$
 $$\eta(s):=\begin{cases}
 X(s) &s\in[0,\tau]\\
 Z(s-\tau)&s\in(\tau,\tau+\delta]\\
 Y(s-\tau-\delta)&s\in(\tau+\delta,\tau+2\delta]\\
 X(s)&s\in(\tau+2\delta,t+2\delta].
 \end{cases}$$
 Then 
 \begin{align*}
 \int_0^t \frac{1}{2}\eta^2(s)&+\frac{1}{2}(b_\alpha(\eta(s))-\dot \eta(s))^2ds + \phi(\eta(s))\\
 &=\int_0^t \frac{1}{2}X^2(s)+\frac{1}{2}(b_\alpha(X(s))-\dot X(s))^2ds+\phi(X(t))\\
 &+\int_0^\delta\frac{1}{2} Y(s)^2+\frac{1}{2}(b_\alpha(Y(s))-\dot Y(s))^2ds\\
 &+\int_0^\delta \frac{1}{2}Z(s)^2+\frac{1}{2}(b_\alpha(Z(s))-\dot Z(s))^2ds\\
 &<u^\alpha(x,t)+\varepsilon+2C^\alpha_1(R)\delta\\
 &\leq u^\alpha(x,t)+\varepsilon+2\bar C^\alpha_1(R)\delta<u^\alpha(x,t)+2\varepsilon.
 \end{align*}
 On the other hand, we have
 $$d^\alpha(x,0)\leq \int_0^{\tau+\delta}\frac{1}{2} \eta^2(s)+\frac{1}{2}(b_\alpha(\eta(s))-\dot \eta(s))^2)ds $$
 and
 $$\phi(0)\leq  \int_{\tau+\delta}^{t+2\delta} \frac{1}{2}\eta^2(s)+\frac{1}{2}(b_\alpha(\eta(s)-\dot \eta(s))^2)ds+\phi(\eta(t+2\delta)).$$
 Therefore 
 $$v^\alpha(x)=d^\alpha(x,0)+\phi^\alpha(0)\leq u^\alpha(x,t)+2\varepsilon\,.$$
 This, together with \eqref{lower} and the arbitrariness of $\varepsilon$, concludes the proof. 
 \begin{flushright}
 $\square$
 \end{flushright}

\section{Numerical approximation and simulations}\label{s5}
In this section, we introduce a numerical scheme for solving the Hamilton-Jacobi equation \eqref{hj}, and we explore some properties of the numerical solution to confirm the results presented in the previous sections. In particular, we analyze its asymptotic behavior in time, in comparison with the solution of the stationary problem \eqref{hjstat}. Then, we employ a simple Euler integrator to build optimal trajectories for the control problem \eqref{value}, and we show the results in different scenarios.

For the reader's convenience, we recall here the Hamilton-Jacobi equation \eqref{hj} together with the boundary condition in space at $x=0$, provided by Lemma \ref{l28}:
$$
\left\{
\begin{array}{ll}
u_t+H(x,Du)=0 & (x,t)\in [0,+\infty)\times[0,+\infty)\\
u(x,0)=\phi(x)     &  x\in [0,+\infty)\\
u(0,t)=\phi(0)     &  t\in [0,+\infty)\,,\\
\end{array}
\right.
$$
where the Hamiltonian $H:[0,+\infty)\times \RR\to \RR$, given by
$$
H(x,p)=(-b_\alpha(x)+\frac{1}{2}p)p-\frac{1}{2}x^2\,,
$$
is rewritten in a form suitable for the discretization.  
We approximate the unbounded space-time domain by means of a rectangle $[0,x_{\max}]\times[0,T]$, for sufficiently large real numbers $x_{\max}>0$ and $T>0$, and we introduce a uniform grid with nodes $(x_i,t_n)=(i\Delta x,n\Delta t)$ for $i=0,\dots,N_x$ and $n=0,\dots,N_t$, where $\Delta x=\frac{x_{\max}}{N_x}$ and $\Delta t=\frac{T}{N_t}$ denote space and time steps respectively, and $N_x, N_t$ are given integers. 

Moreover, we denote the approximations of $u$, $b_\alpha$, $\phi$ on the grid respectively by $U_i^n\simeq u(x_i,t_n)$, $B_{\alpha,i} \simeq  b_\alpha(x_i)$, $\Phi_i \simeq \phi(x_i)$, and we collect them in the vectors $U^n=(U_0^n,\dots,U_{N_x}^n)$, $B_\alpha=(B_{\alpha,0},\dots,B_{\alpha,N_x})$, $\Phi=(\Phi_0,\dots,\Phi_{N_x})$.  
Then, we introduce finite differences for approximating space and time derivatives. In particular, discretization in space requires some care, in order to define a numerical Hamiltonian $H^\sharp$ which correctly approximates viscosity solutions of the equation. More precisely, using forward/backward differences, 
we introduce the following two-sided approximation of $Du$, 
$$
DU_i=(D_LU_i,D_RU_i):=\left(\frac{U_i-U_{i-1}}{\Delta x},\frac{U_{i+1}-U_{i}}{\Delta x} 
\right)\,,
$$
for $i=1,\dots,N_x-1$, 
and we set 
$$
H^\sharp(x_i,DU_i):=\Big(-B_{\alpha,i}+\frac12 D_LU_i\Big)^+ D_LU_i+\Big(-B_{\alpha,i}+\frac12 D_RU_i\Big)^- D_RU_i\,,
$$
which selects the gradient components in an upwind fashion, according to the sign of $-B_\alpha+\frac12 DU$,
where $(\cdot)^+=\max\{\cdot,0\}$, $(\cdot)^-=\min\{\cdot,0\}$ denote respectively the positive and negative parts of their arguments (for further details, we refer the interested reader to \cite{sethian,falfer}).

Finally, we employ a forward difference in time, and we end up with the following explicit time-marching scheme:
$$
\left\{
\begin{array}{ll}
U_i^{n+1}=U_i^n-\Delta t H^\sharp(x_i,DU_i^n) & i=1,\dots,N_x,\,\, n=0,\dots,N_t\\
U_i^0=\Phi_i     &  i=0,\dots,N_x\\
U_0^n=\Phi_0     &  n=0,\dots,N_t\,.
\end{array}
\right.
$$
We remark that the forward difference $D_RU_i$ is not defined for $i=N_x$, corresponding to the point $x=x_{\max}$ at the right boundary. Since the feedback control for the underlying control problem is given by $\xi=-Du$, we can argue that, for $x_{\max}$ sufficiently large (e.g. greater than the endemic population $E=N(1-\frac{1}{\rho})$ if $\rho>1$), the optimal dynamics in \eqref{controlya} pushes to the left, and towards the origin, the trajectory starting from $x_{\max}$. This implies that $Du(x_{\max},t)\ge 0$ for all $t\in(0,T)$, and we also have $b_{\alpha}(x_{\max})\le 0$. We conclude that $(-b_\alpha(x_{\max})+\frac12 Du(x_{\max},t))^-=0$, hence 
we can always select for the node $i=N_x$ just the contribution given by the backward approximation $D_LU_i$ .

We also remark that, due to the definition of $b_\alpha$, the right hand side of the above scheme becomes larger and larger as $x_{\max}$ increases. This requires a severe CFL restriction on the discretization steps $\Delta t$ and $\Delta x$, in order to preserve stability. 

Once the solution has been computed, we reconstruct the following semi-discrete feedback control, by merging the two components of $DU$ and interpolating their values in space,  via a linear interpolation operator $\mathbb{I}$:
$$\xi^n(\cdot)=-\mathbb{I}[ (D_LU^n)^+ + (D_RU^n)^-](\cdot)\,.$$
Then, we build the optimal trajectories by integrating \eqref{controlya}, using a simple forward Euler scheme
$$
\begin{cases}
y^{n+1}=y^n+\Delta t\left(b_\alpha(y^n)+\xi^n(y^n)\right)\\
y^0=x\,.\end{cases}
$$

Now, let us setup the parameters for the numerical experiments. We choose the space domain size $x_{\max}=4$, the final time $T=5$, and $N_x=200$, $N_t=4000$ nodes in space and time respectively. Moreover, we choose $M(\alpha)\equiv 1$, $N=\frac94$ and  $\gamma=1$ as model parameters defining the advection speed $b_\alpha$ in \eqref{ordSIS}. On the other hand, $\alpha$ and $\rho$ (and accordingly $\beta=\gamma \rho$) will be set differently for each test.

We start by choosing the exit cost $\phi(x)=x$ and a reproduction factor $\rho=\frac{3}{2}$ such that the corresponding endemic population $E=\frac{3}{4}$ is a stable equilibrium (attractive in the domain $[0,x_{\max}]$) for the uncontrolled system, see Proposition \ref{equilibrium}. In Figure \ref{Test1}, we show the results obtained for $\alpha=1$ and $\alpha=\frac{1}{2}$ at different times. In the left panels, we report the value functions compared to $\phi$, in the right ones the corresponding optimal controls. 
\begin{figure}[!h]
	\centering{
	\includegraphics[width=0.45\textwidth]{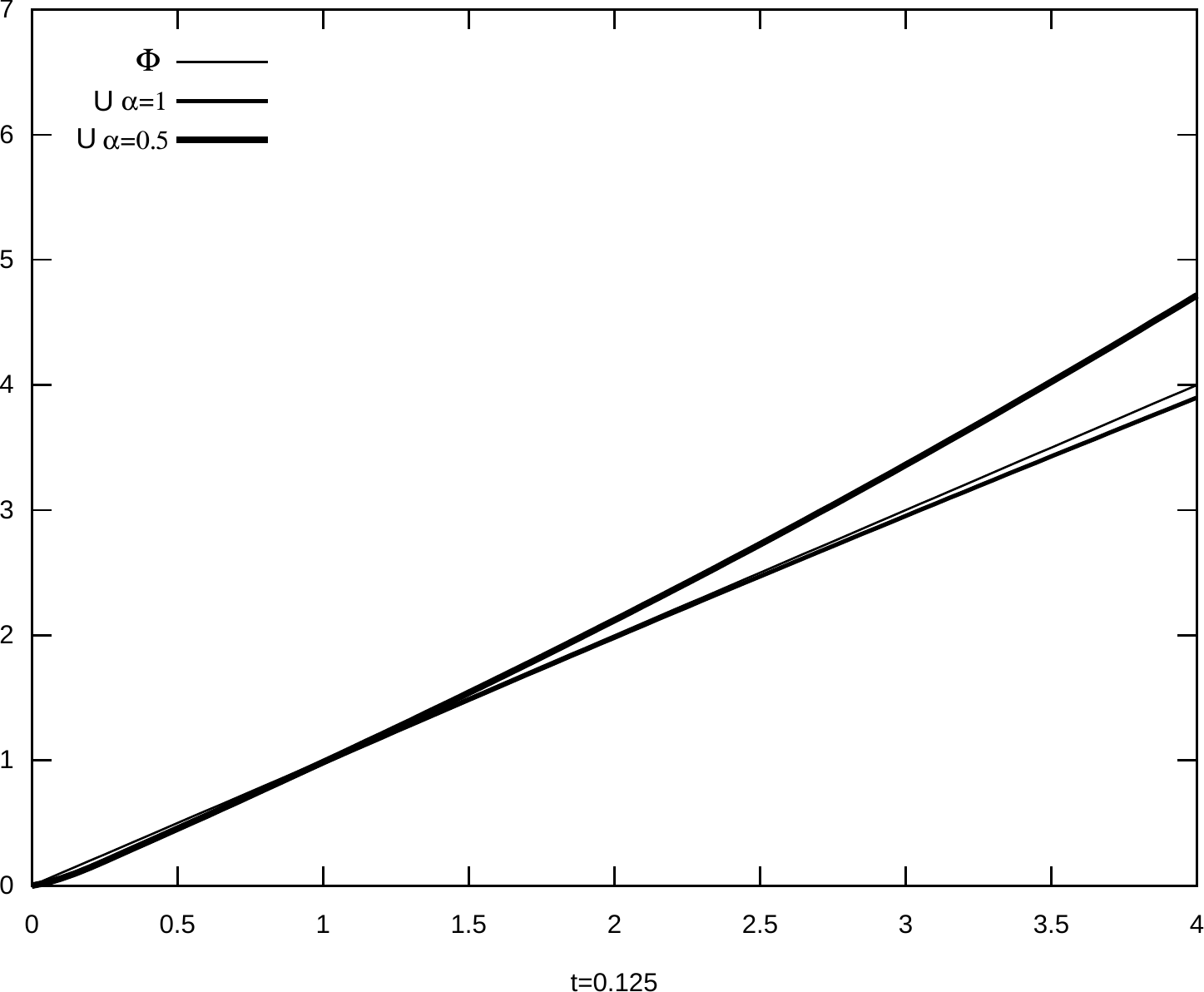} 
	\includegraphics[width=0.45\textwidth]{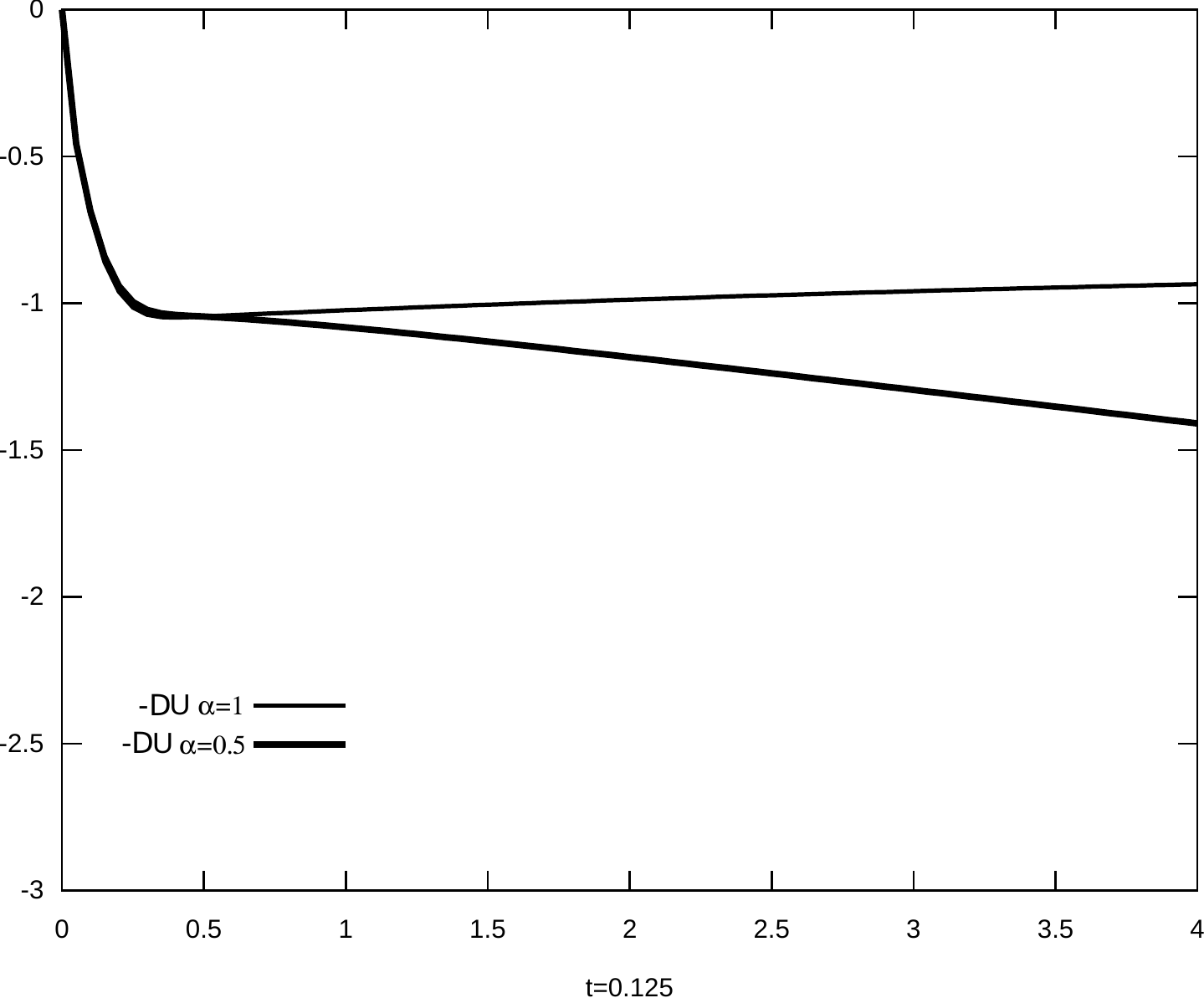}\\
	\includegraphics[width=0.45\textwidth]{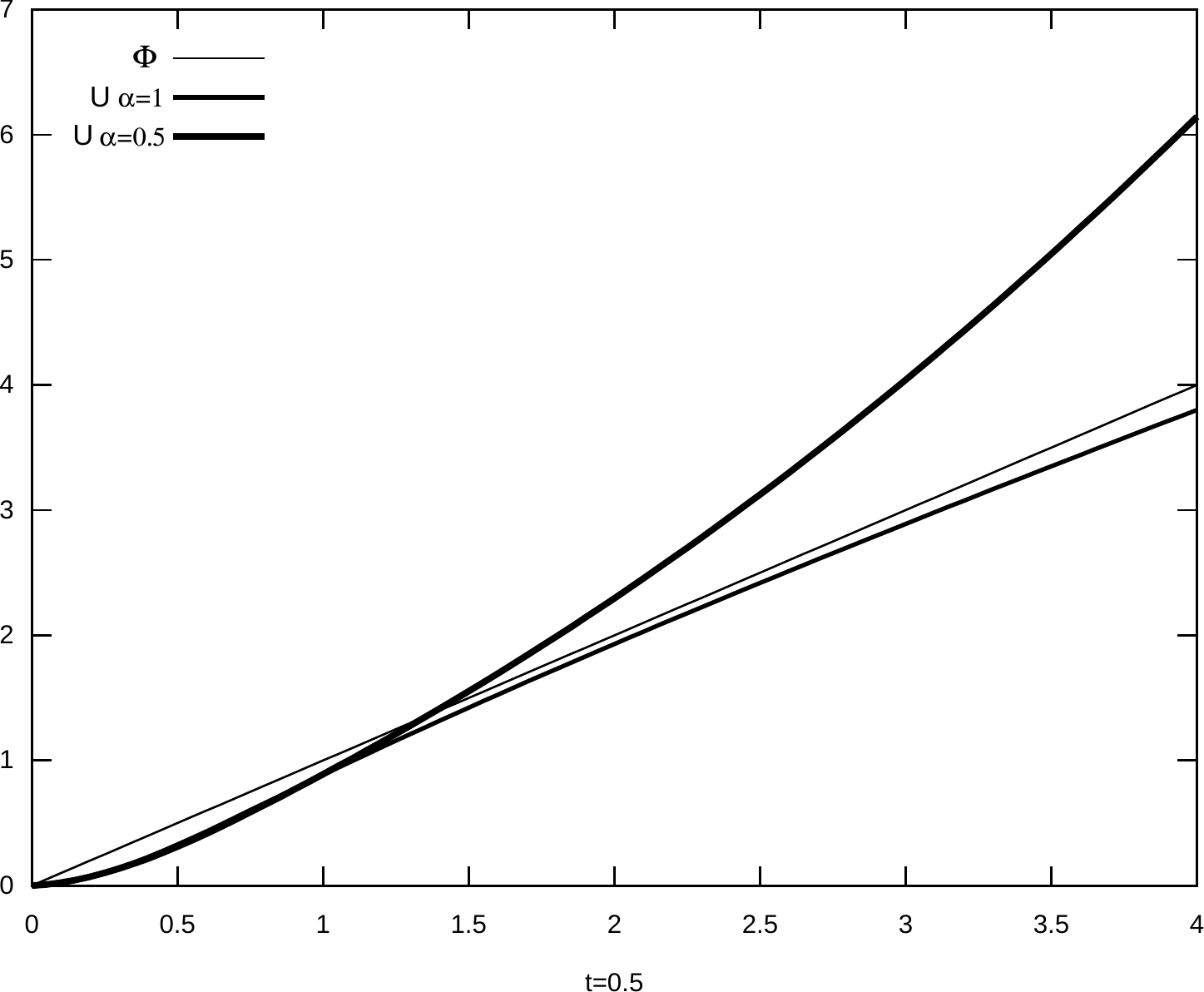} 
	\includegraphics[width=0.45\textwidth]{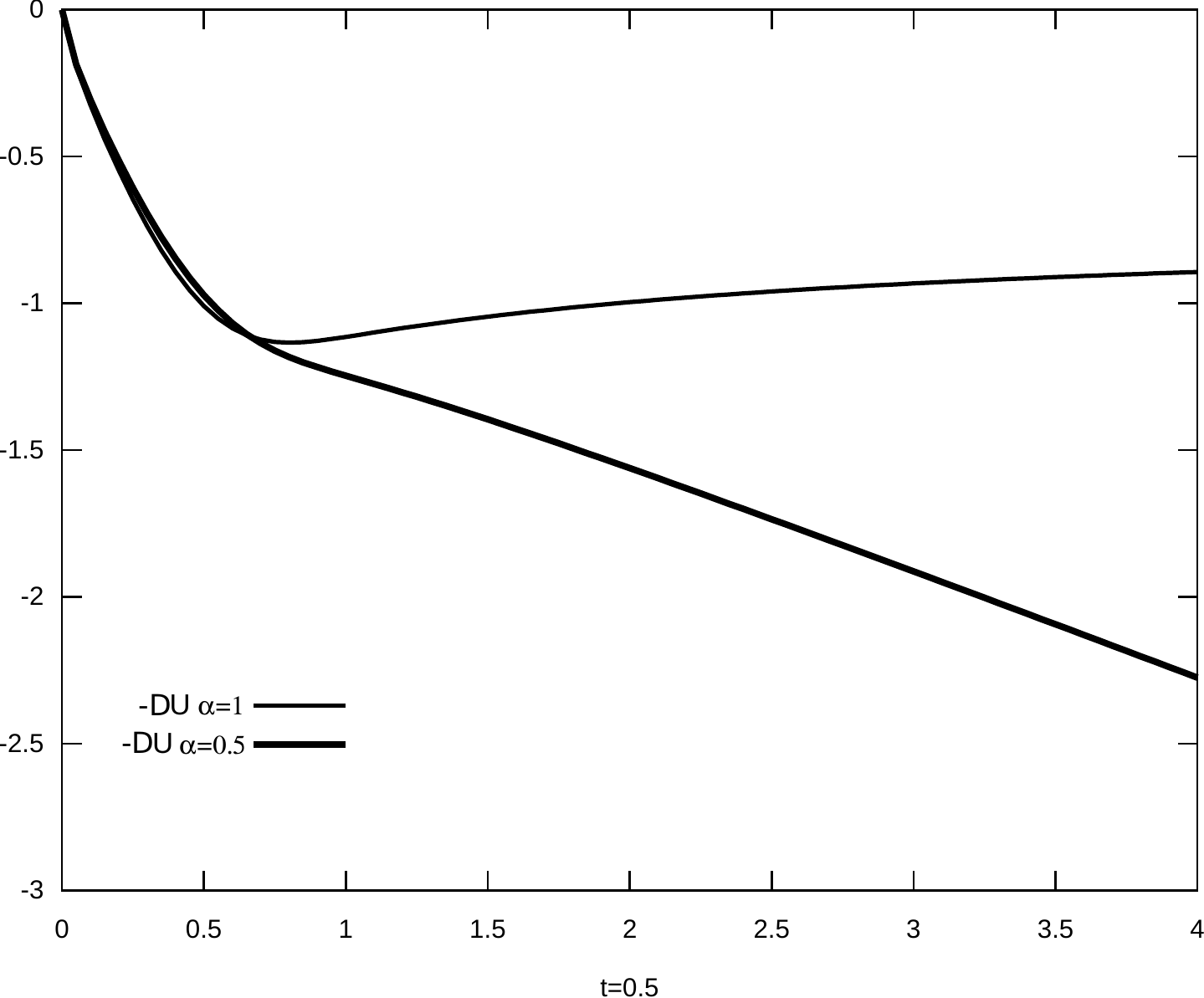}\\
	\includegraphics[width=0.45\textwidth]{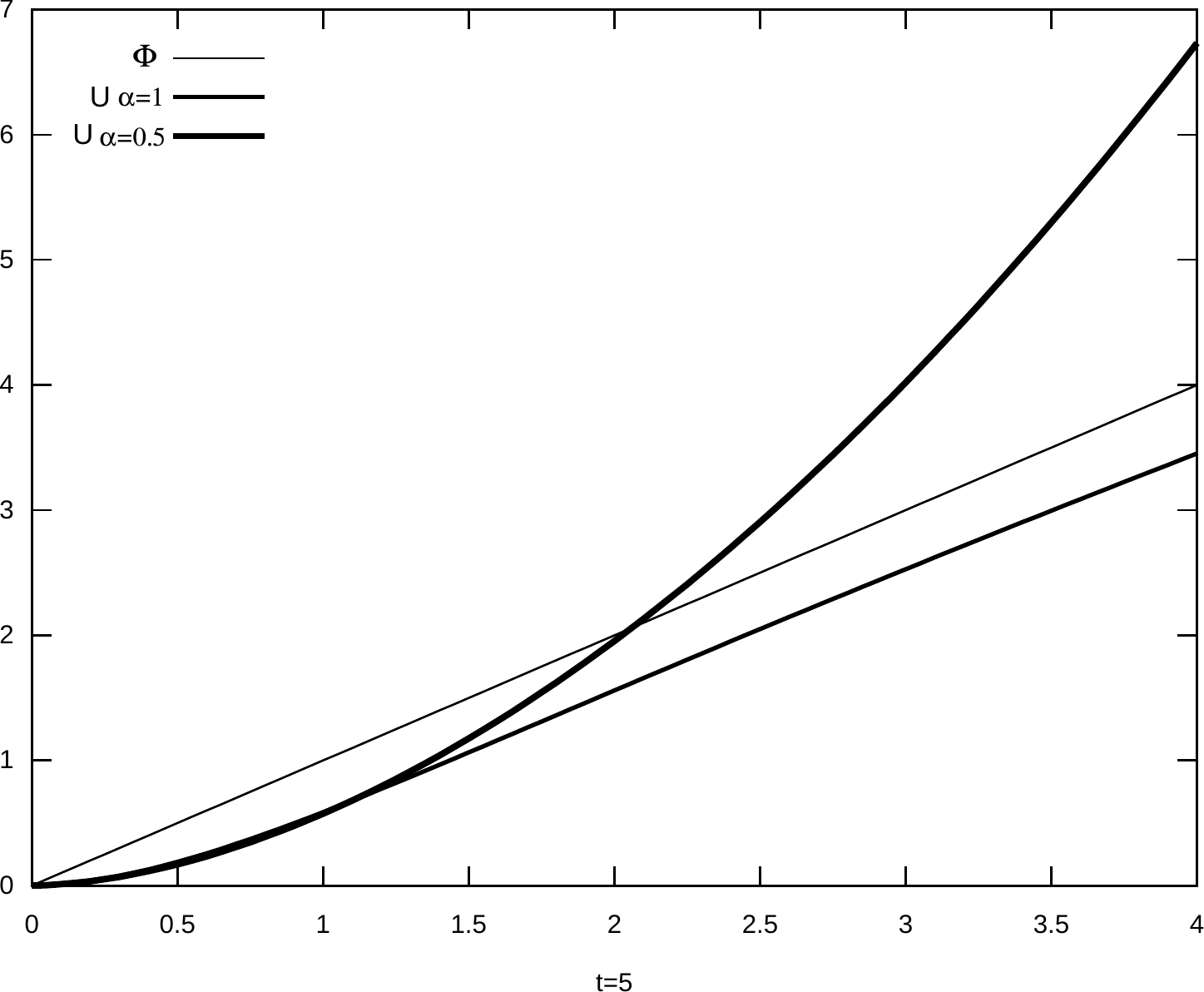} 
	\includegraphics[width=0.45\textwidth]{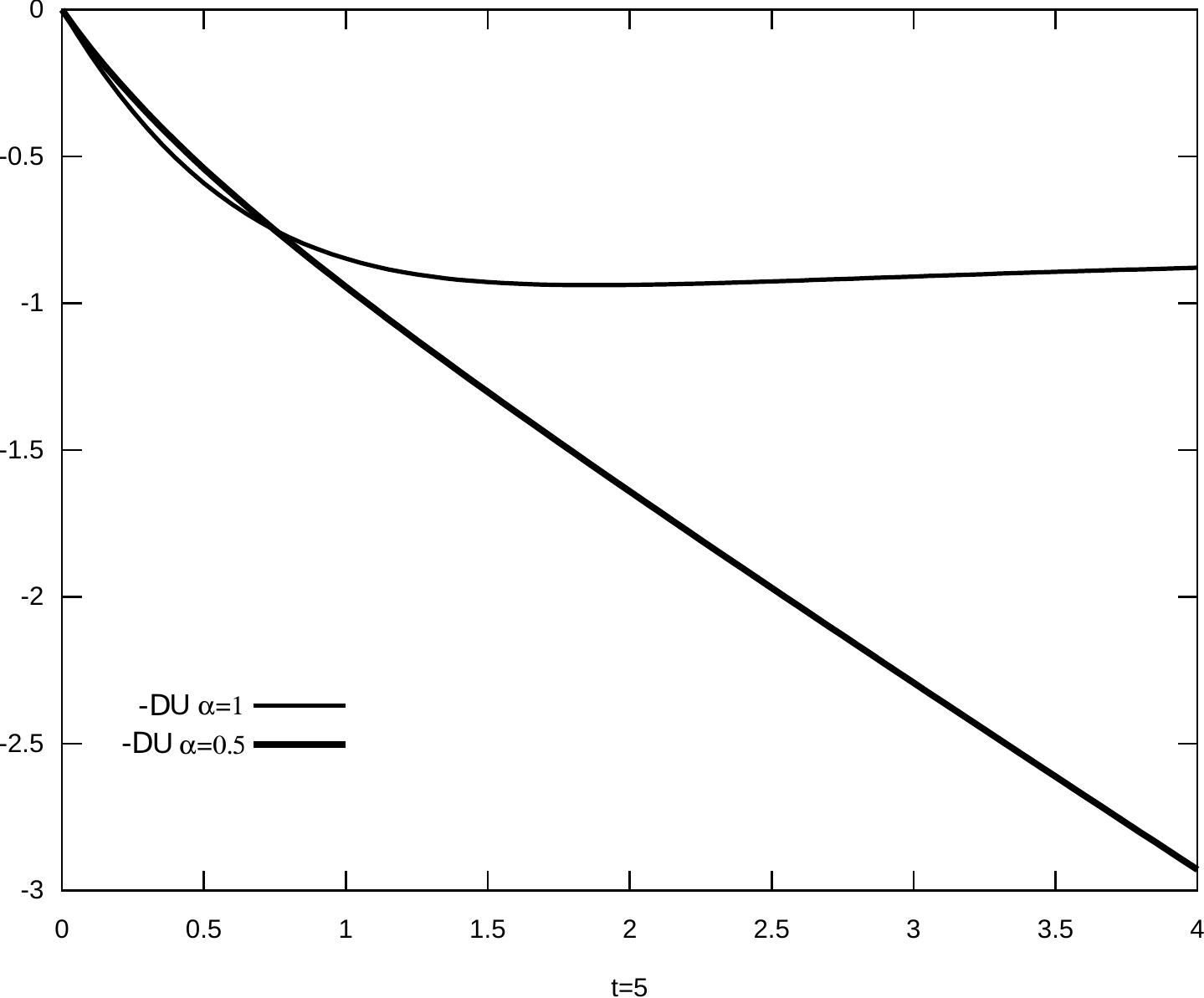}}
\caption{Value functions (left panels) and optimal controls (right panels) at different times for $\alpha=1$, $\alpha=\frac12$ and $\rho=\frac{3}{2}$.}\label{Test1}
\end{figure}

Note that the time horizon $T=5$ has been set large enough to reveal the asymptotic behavior in time of the solutions, namely their convergence, up to machine error, to stationary regimes. Similarly, the space boundary $x_{\max}$=4 is large enough to distinguish the growth of the solutions for $x\to+\infty$. We observe a linear behavior for the case $\alpha=1$, and a quadratic behavior for the case $\alpha=\frac12$. This can be better appreciated looking at the corresponding optimal controls, and it is confirmed by the simulation in Figure \ref{Test1-asymptotics}, in which we show, for $0<\alpha\le 1$, the asymptotic behavior of the $L^\infty$ norms in space of $U$ and $DU$ at the final time (achieved by monotonicity at $x_{\max}$) as both $x_{\max},\,T\to+\infty$. 
\begin{figure}[!h]
	\centering
	\includegraphics[width=0.46\textwidth]{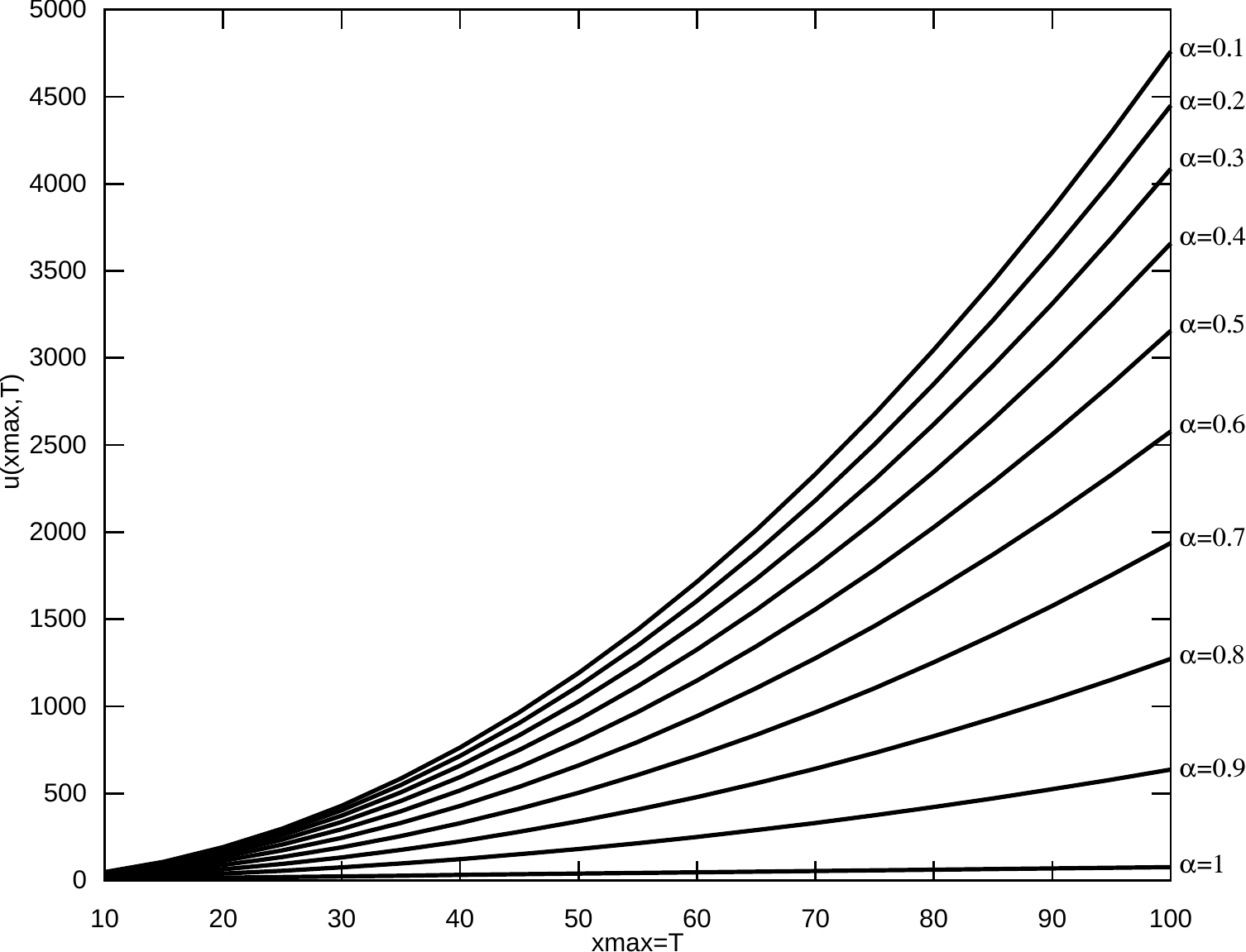} 
	\includegraphics[width=0.45\textwidth]{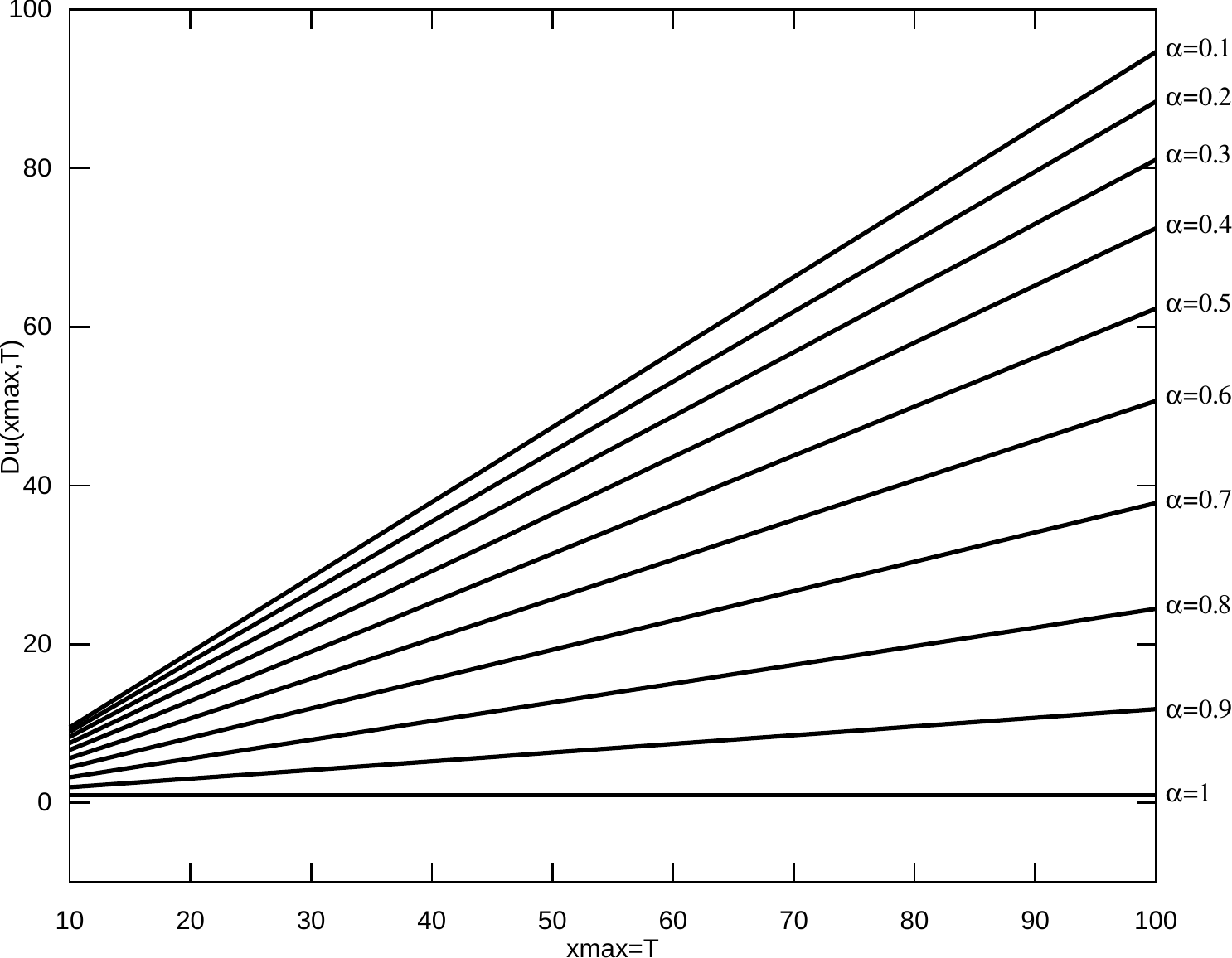}
\caption{$L^\infty$ space norms (at final time) of the value functions (left) and of their gradients (right) as $x_{\max},\,T\to+\infty$ for different values of $\alpha$ ranging in $[0,1]$.}\label{Test1-asymptotics}
\end{figure}
In particular, we find out that $\alpha=1$ is the only value that produces a globally Lipschitz continuous solution. A rigorous proof of this statement is still under investigation. 

In Figure \ref{Test1-trj}, we compare some optimal trajectories obtained for $\alpha=1$ (top panels) and $\alpha=\frac12$ (bottom panels).  In each plot, we report the endemic population $E$ (dashed line), the uncontrolled/controlled trajectories (bold lines), and the corresponding optimal controls (thin lines). Moreover, we choose two different initial data for the dynamics \eqref{controlya}, $x=0.5$ and $x=1.25$, respectively below and above $E$. 
\begin{figure}[!h]
\begin{tabular}{cc}
	\includegraphics[width=0.45\textwidth]{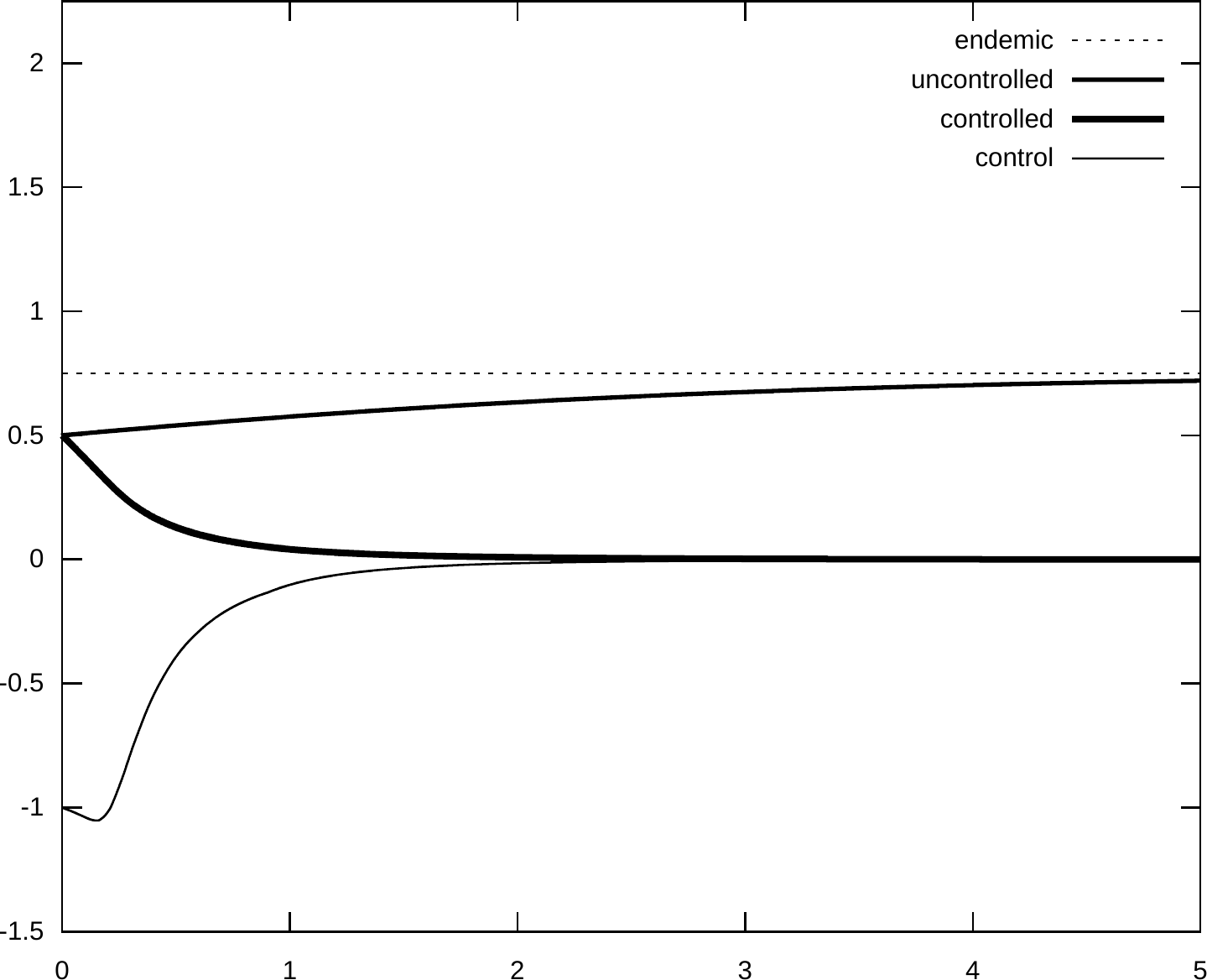} 
	&
	\includegraphics[width=0.45\textwidth]{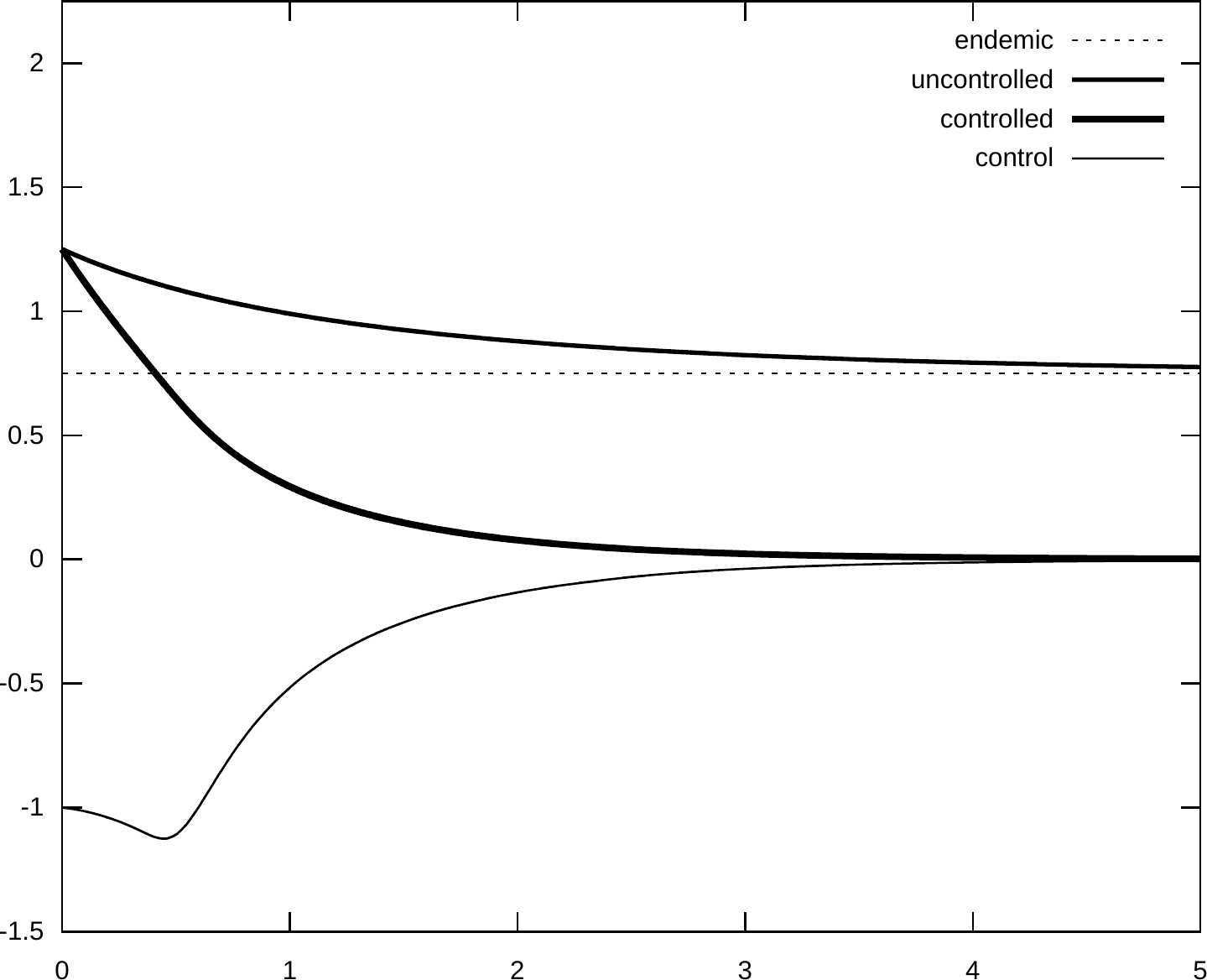}\\
	$\alpha=1,\,x=0.5$ & $\alpha=1,\,x=1.25$\\\\
	\includegraphics[width=0.45\textwidth]{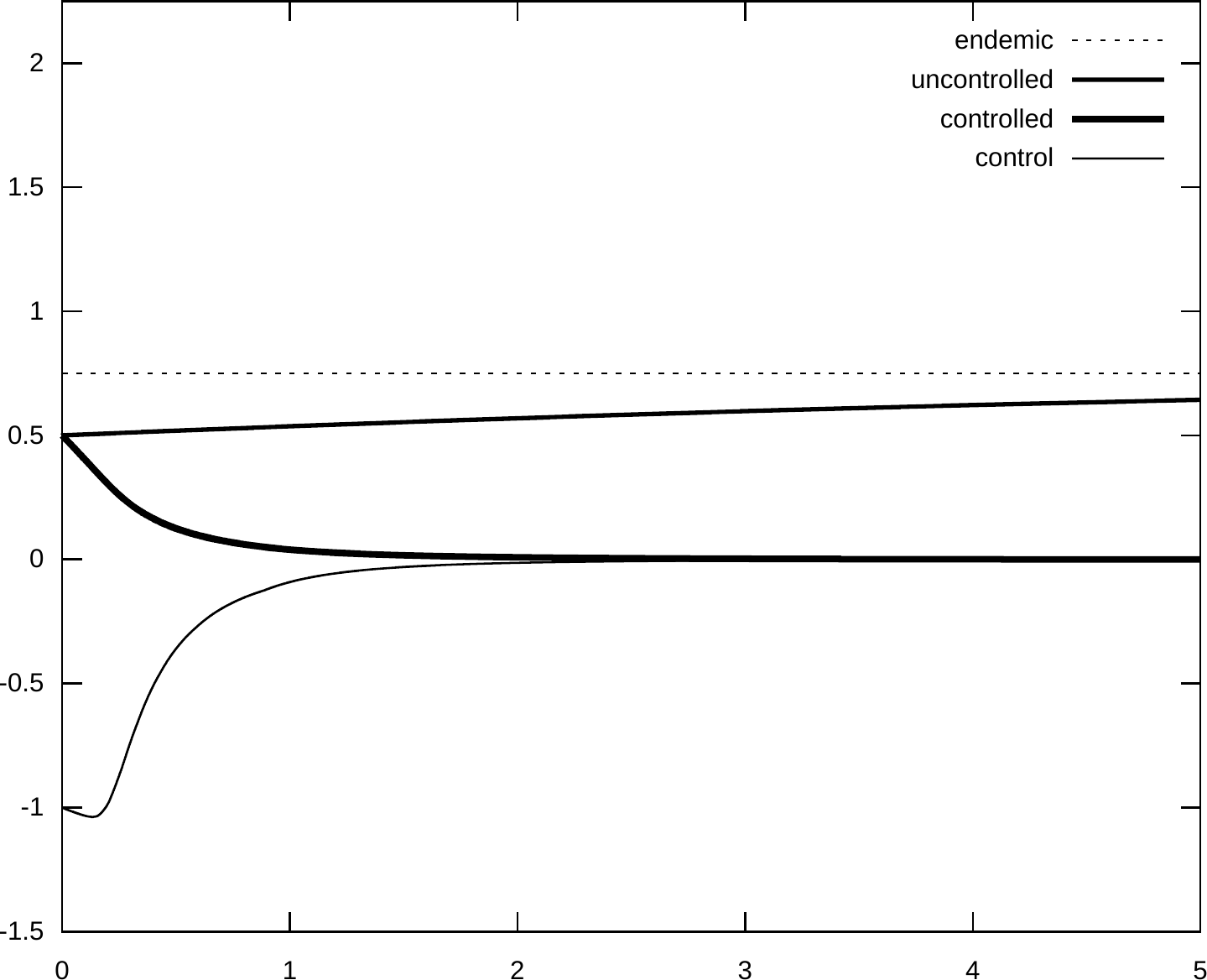} &
	\includegraphics[width=0.45\textwidth]{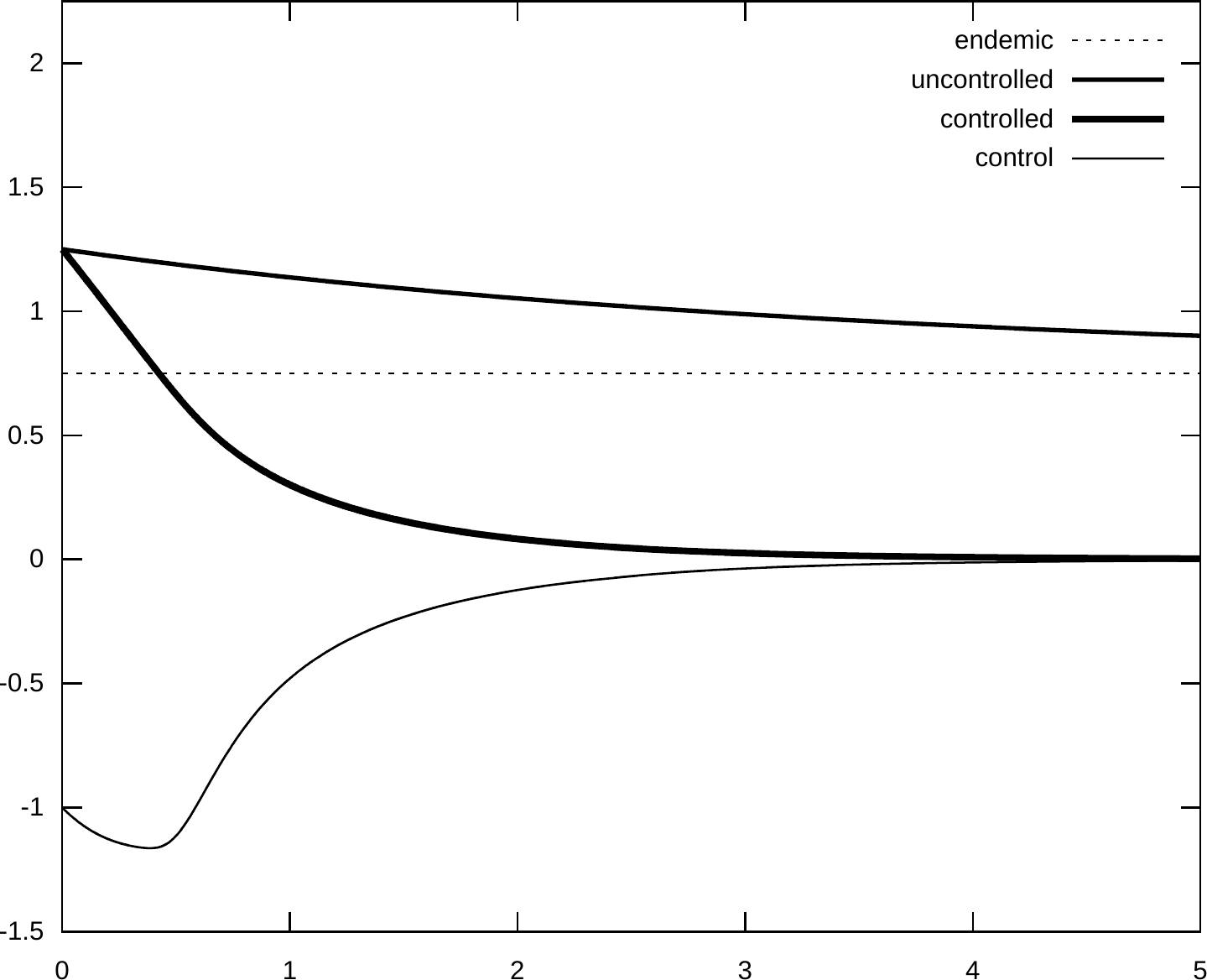}\\
		$\alpha=\frac12,\,x=0.5$ & $\alpha=\frac12,\,x=1.25$\\
	
\end{tabular}
\caption{Optimal trajectories for different fractional orders $\alpha$ and initial data $x$.}\label{Test1-trj}
\end{figure}
 As discussed in the introduction, the fractionary SIS system can be recasted  in the model \eqref{saturatedSIS}, with ordinary derivatives and saturated growth rates. In particular, the growth of infective individuals is softened as $\alpha$ decreases. This effect is apparent in the uncontrolled trajectories. In the same time horizon, we observe that the uncontrolled trajectory approaches $E$ for $\alpha=1$, while for $\alpha=\frac12$ it is ``lazier''  and still far from the endemic value at the final time. 
 On the other hand, we observe that the optimal control always succeeds in steering the system to the origin (the unstable equilibrium in this case). Nevertheless, while the controlled trajectories are quite similar (as their optimal controls) when the evolution starts from $x=0.5<E$, the case $x=1.25>E$ for $\alpha=\frac12$ requires an additional effort to compensate the slower decay of the corresponding dynamics. Indeed, we observe an optimal control with a larger amplitude in the fragment $[0,0.5]$ of the time interval.
 
The numerical results for the case with a reproduction factor $\rho\le 1$ are quite similar to the previous ones, and we omit them for brevity. We just remark that the endemic value now falls out of the space domain ($E\le 0$), while the state $x=0$ is a stable equilibrium for the system, see again Proposition \ref{equilibrium}.  This implies that, for all the initial data in $(0,x_{\max}]$, the corresponding uncontrolled trajectories eventually converge to $x=0$, whereas the controlled ones have a faster decay, in order to optimize the cost functional \eqref{value} for the optimal control problem.

Let us now consider an example with a non smooth exit cost, namely we choose $\phi(x)=\min\{2x+\frac12,6x^2\}$, so to produce a kink in the solution. Moreover, we choose the parameters as in the previous tests, with the exception of the space domain size, that we set to $x_{\max}=2$ in order to achieve a sharper CFL condition and mitigate the numerical diffusion of the scheme. In Figure \ref{kink}, we show the results for the case $\rho=\frac{3}{2}$ and $\alpha=1$. In each plot we report, at different times, the value function compared to $\phi$ and also the optimal control. 
\begin{figure}[!h]
	\centering
	\includegraphics[width=0.45\textwidth]{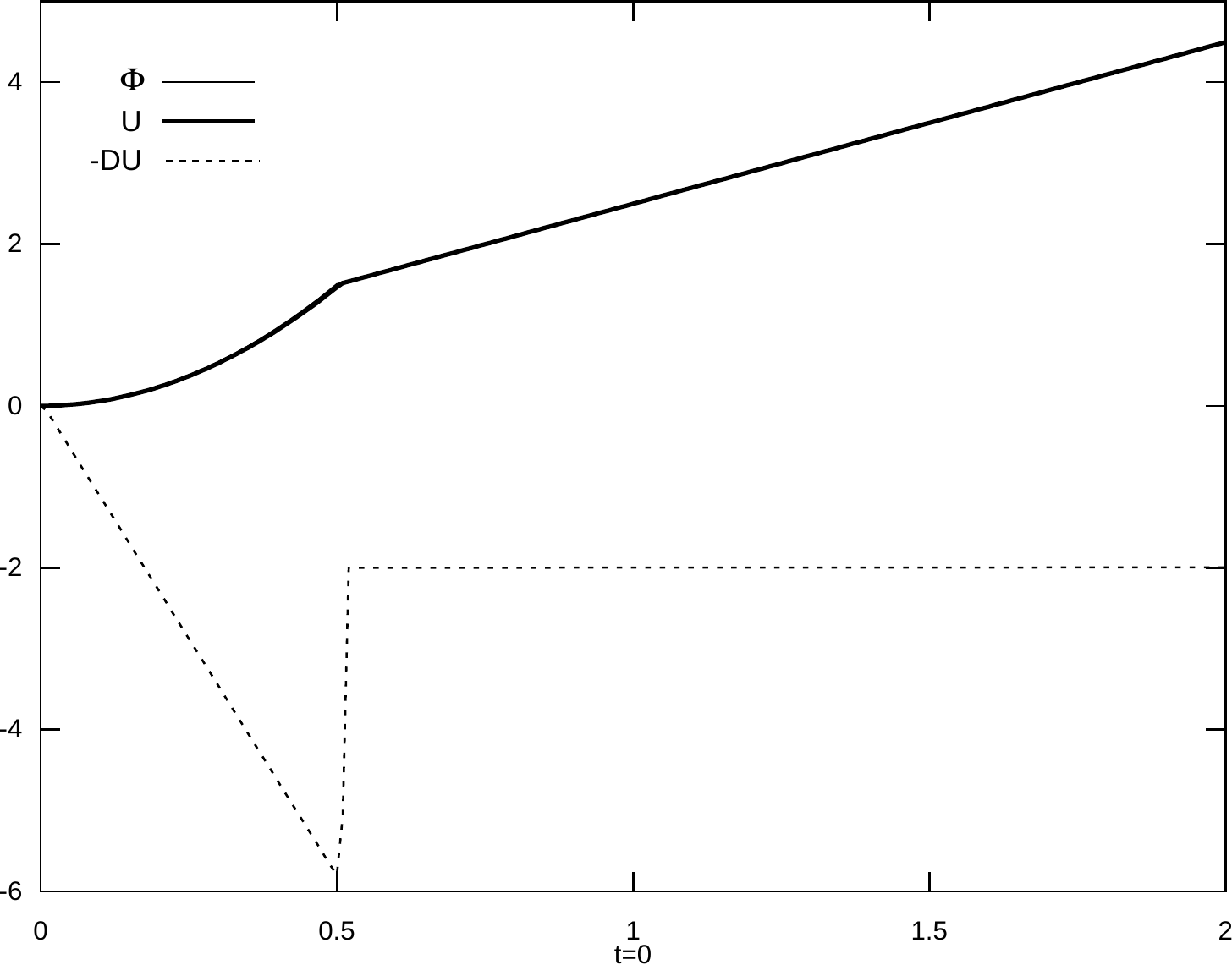} 
	\includegraphics[width=0.45\textwidth]{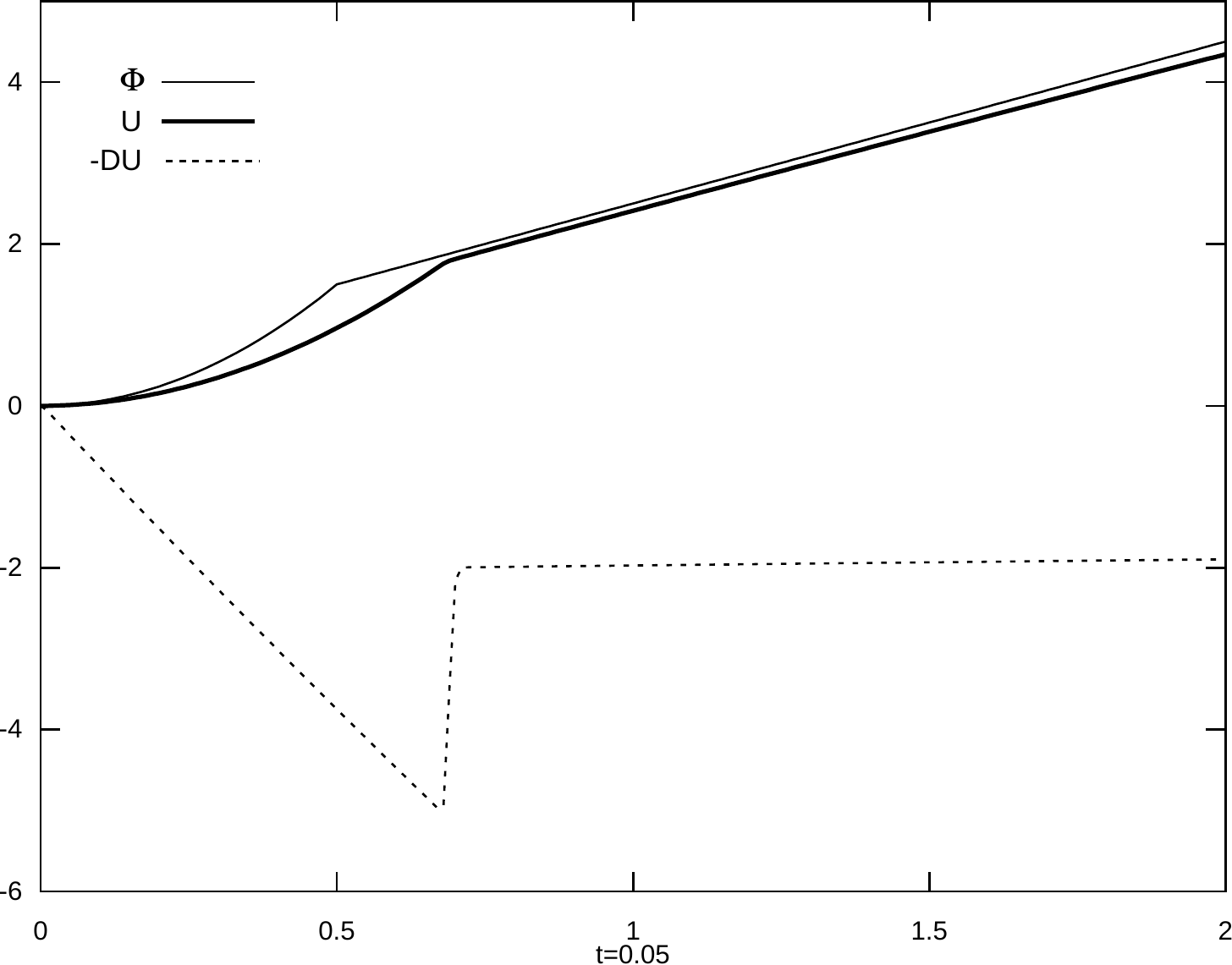}\\
	\includegraphics[width=0.45\textwidth]{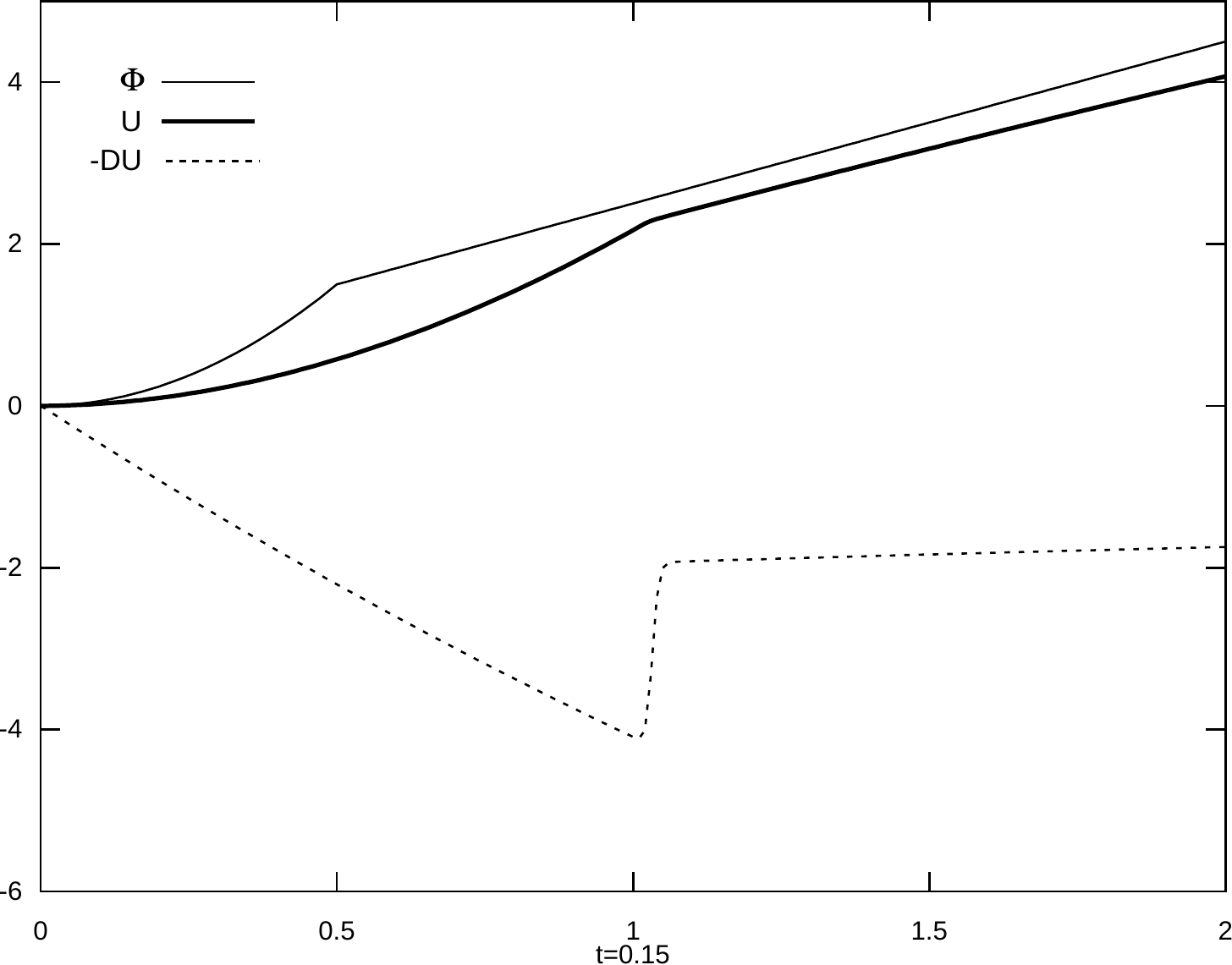} 
	\includegraphics[width=0.45\textwidth]{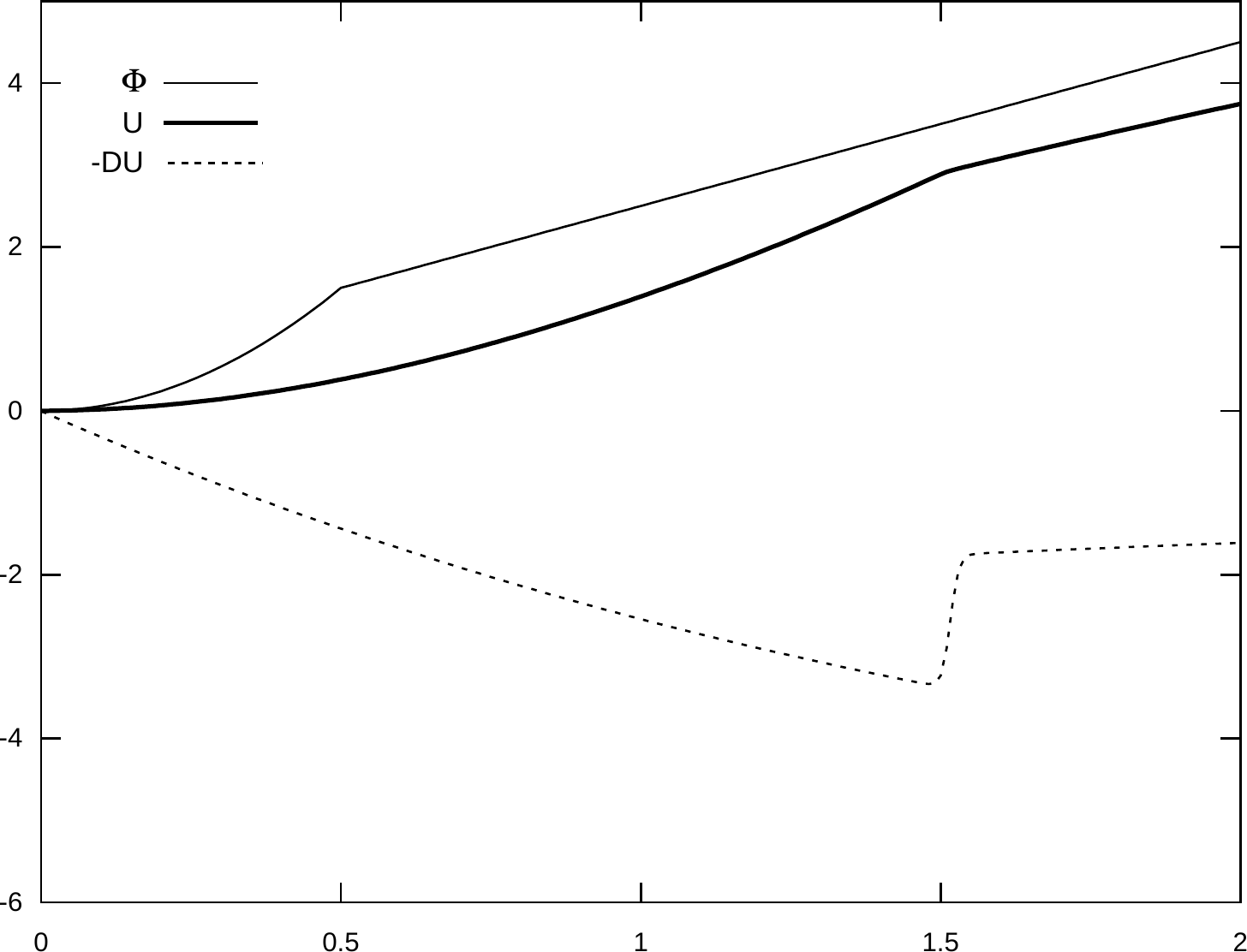}\\
	\includegraphics[width=0.45\textwidth]{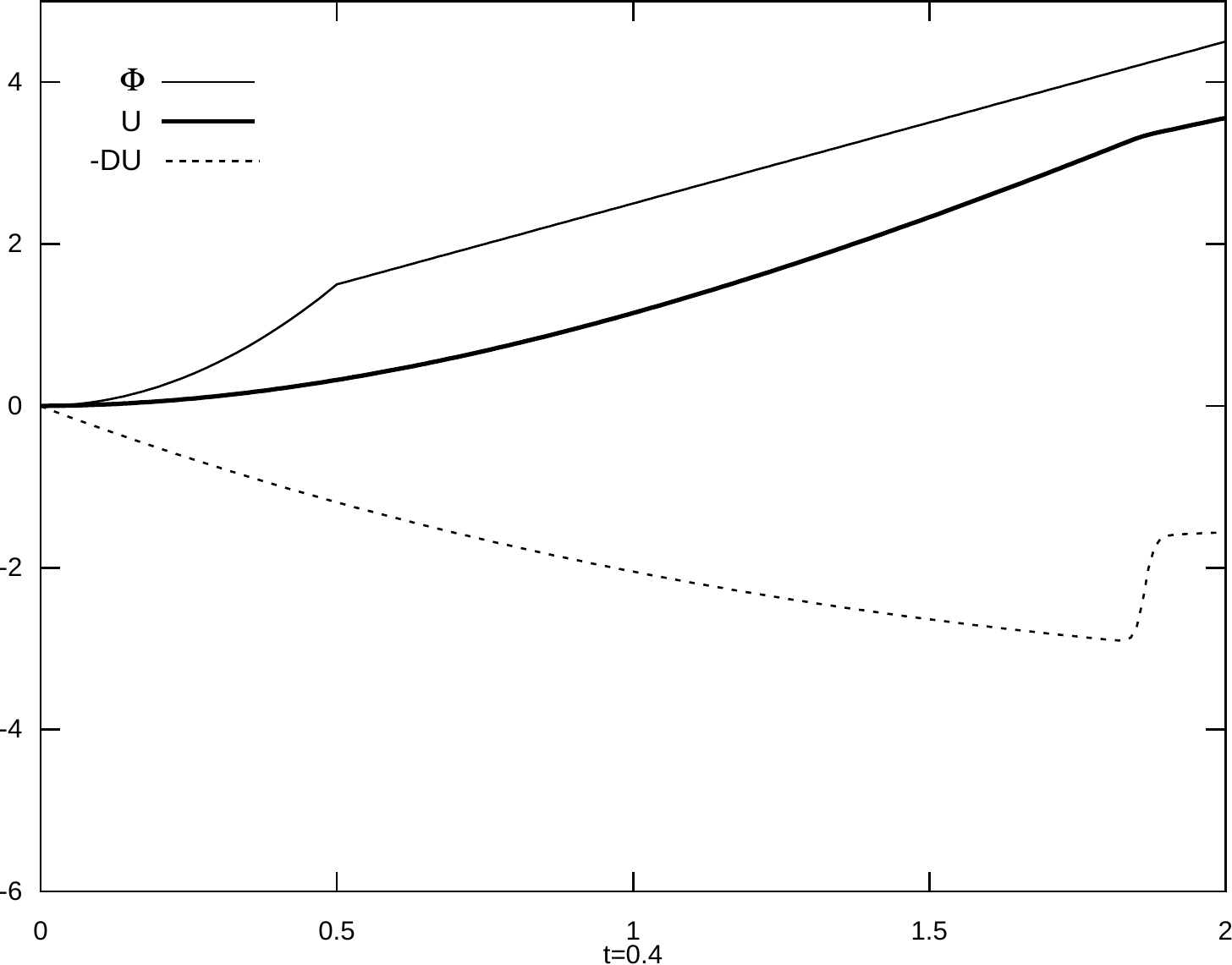} 
	\includegraphics[width=0.45\textwidth]{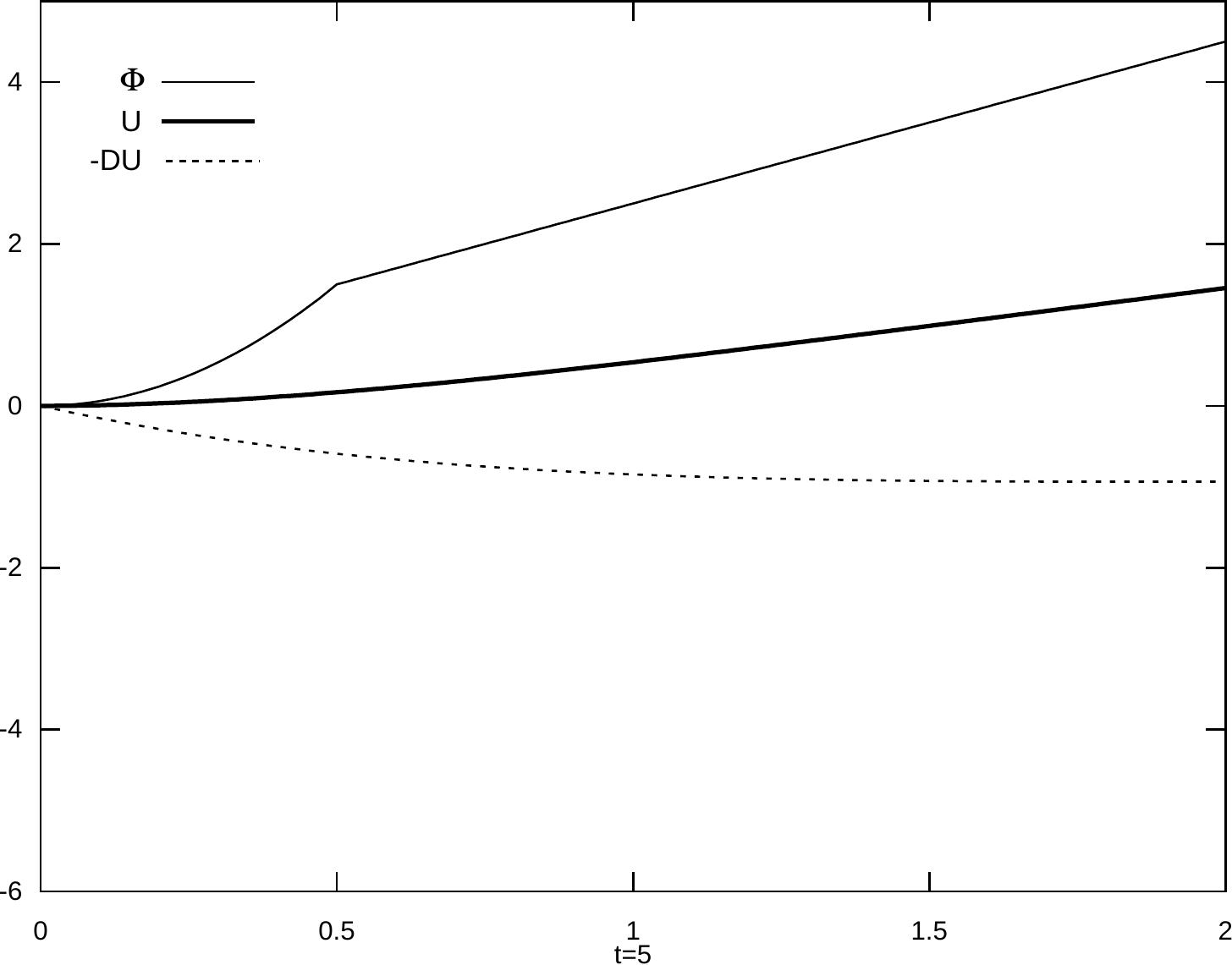}
\caption{Non smooth exit cost, value function and optimal control at different times.}\label{kink}
\end{figure}
We clearly observe that the kink in the solution moves and eventually exits the domain as the time increases. Asymptotically, we obtain a smooth solution as in the previous tests. 

We finally consider the case of a 
smooth exit cost $\phi(x)=x+exp(-40(x-\frac12)^2)$, which corresponds to penalize the final distribution of infective individuals around the point $x=\frac12$ more than for larger values (up to about $x=\frac32$). The results for the case $\rho=\frac{3}{2}$ and $\alpha=1$ are reported in Figure \ref{kink-gen}. We observe that, in the first part of the evolution, the point $x=\frac12$ acts as a barrier, preventing some states of the system to be steered to the desired one $x=0$. More precisely, the local minimizer of $\phi$ (around about $x=0.8$) is more favorable for states beyond this barrier, where the optimal control has a change of sign. This creates a kink in the solution, which starts moving towards the right boundary of the domain only at a later time. 
\begin{figure}[!h]
	\centering
	\includegraphics[width=0.45\textwidth]{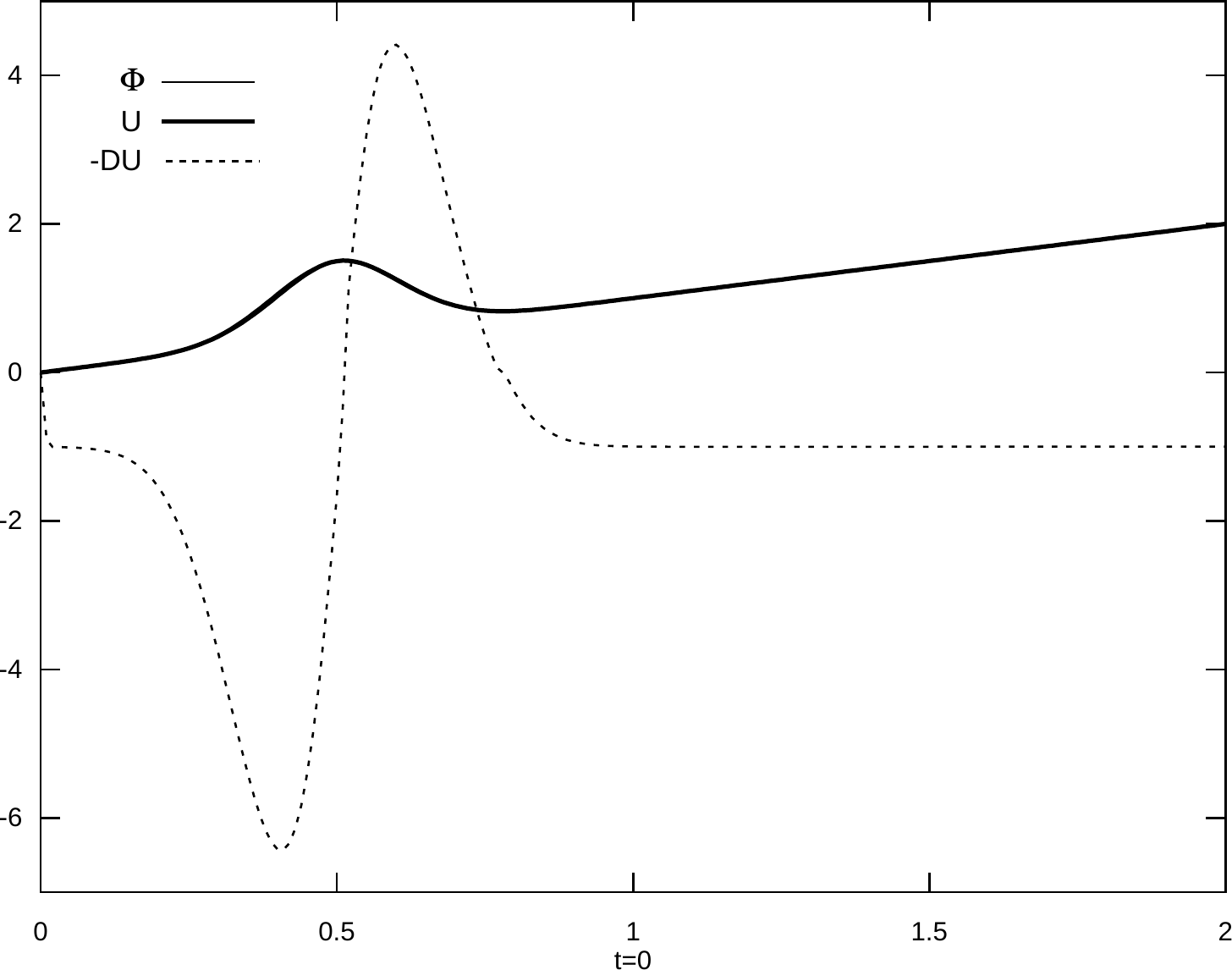} 
	\includegraphics[width=0.45\textwidth]{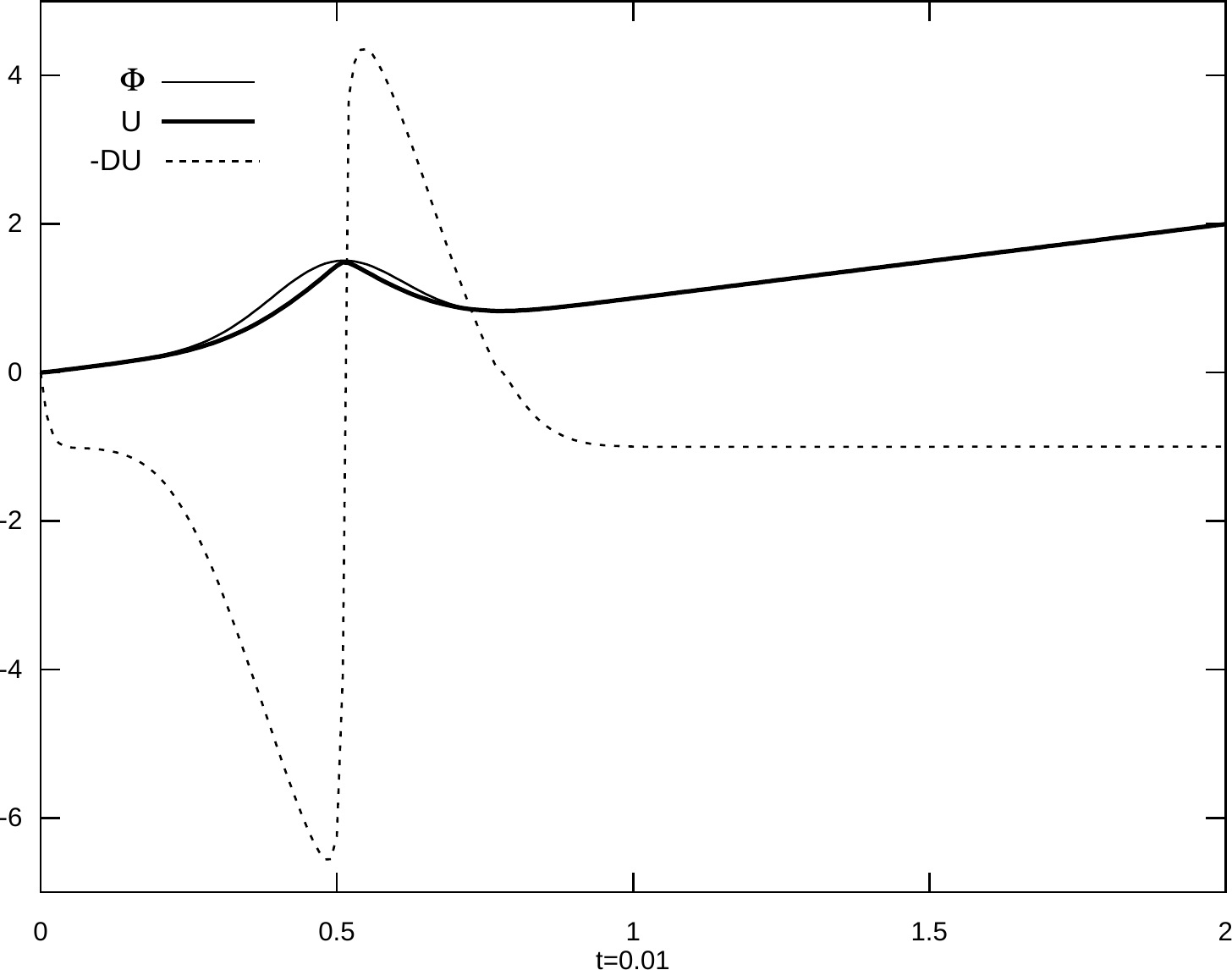}\\
	\includegraphics[width=0.45\textwidth]{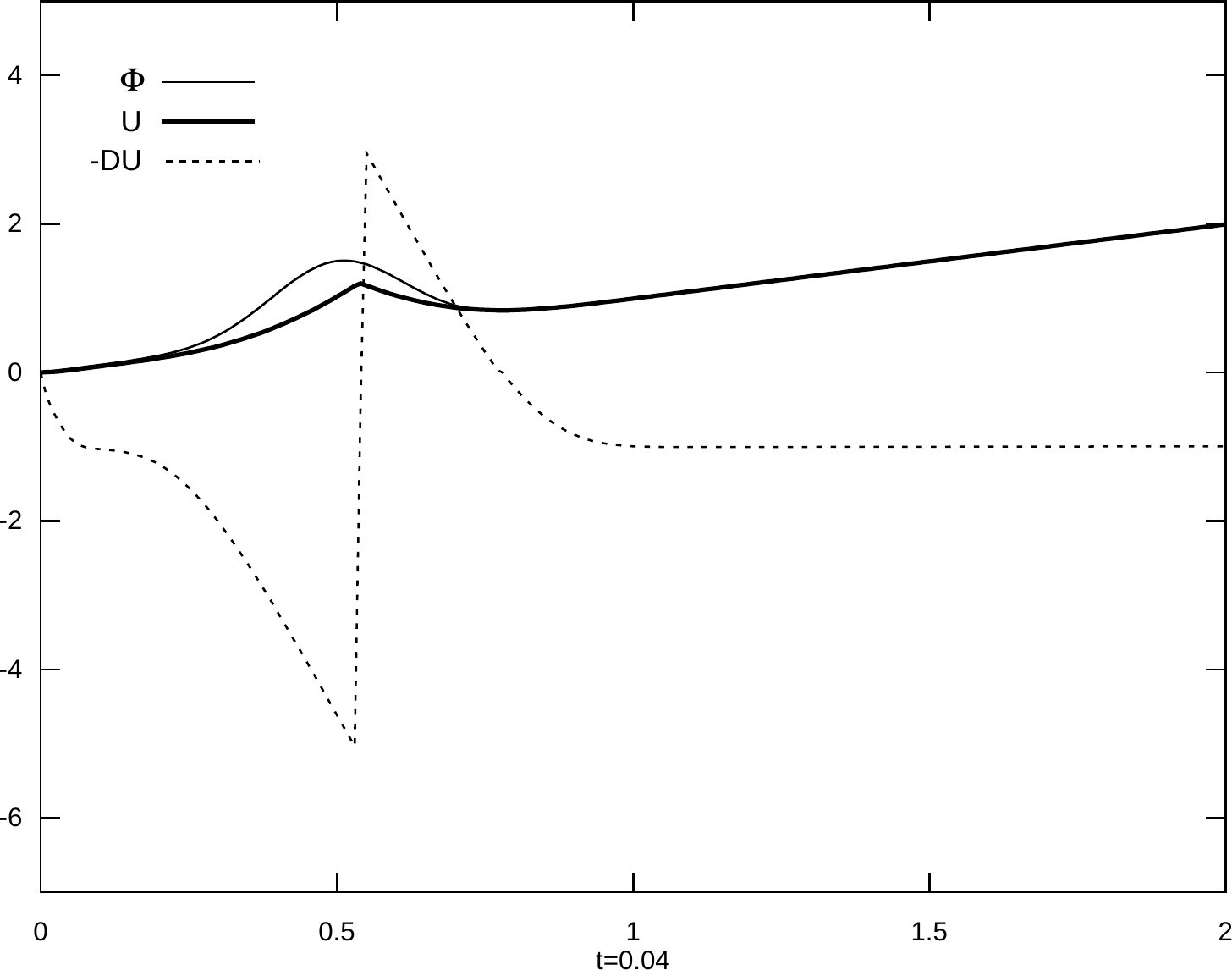} 
	\includegraphics[width=0.45\textwidth]{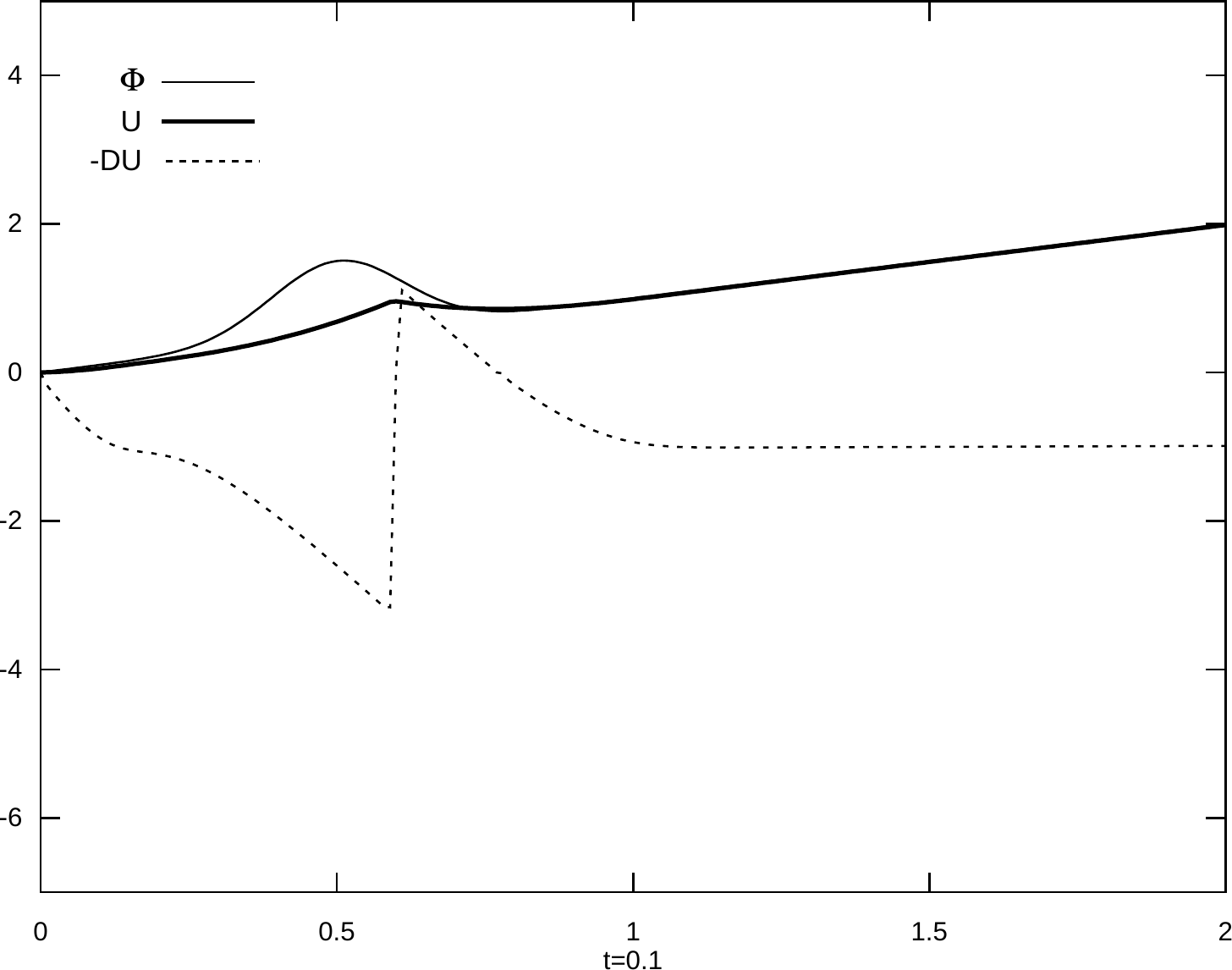}\\
	\includegraphics[width=0.45\textwidth]{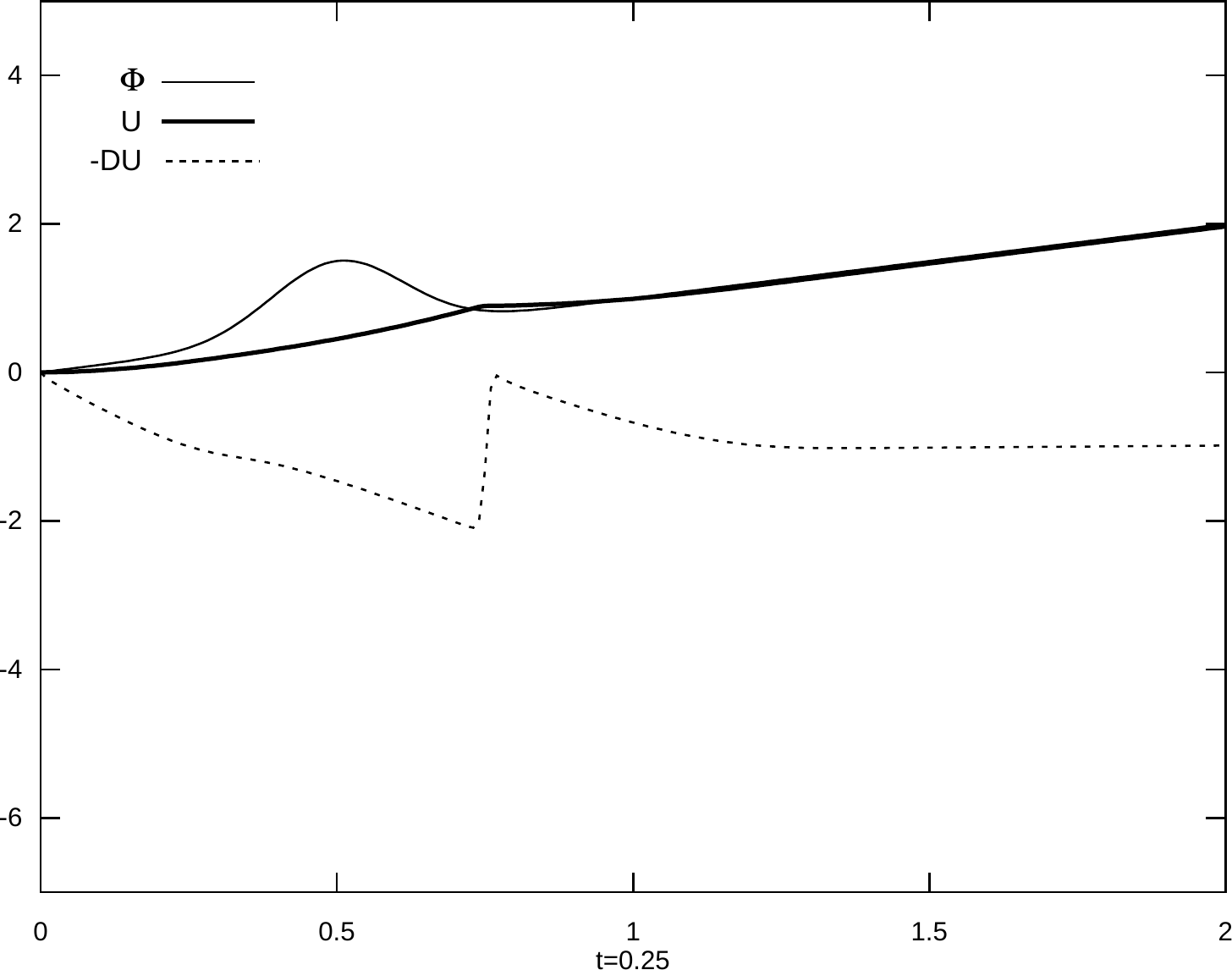} 
	\includegraphics[width=0.45\textwidth]{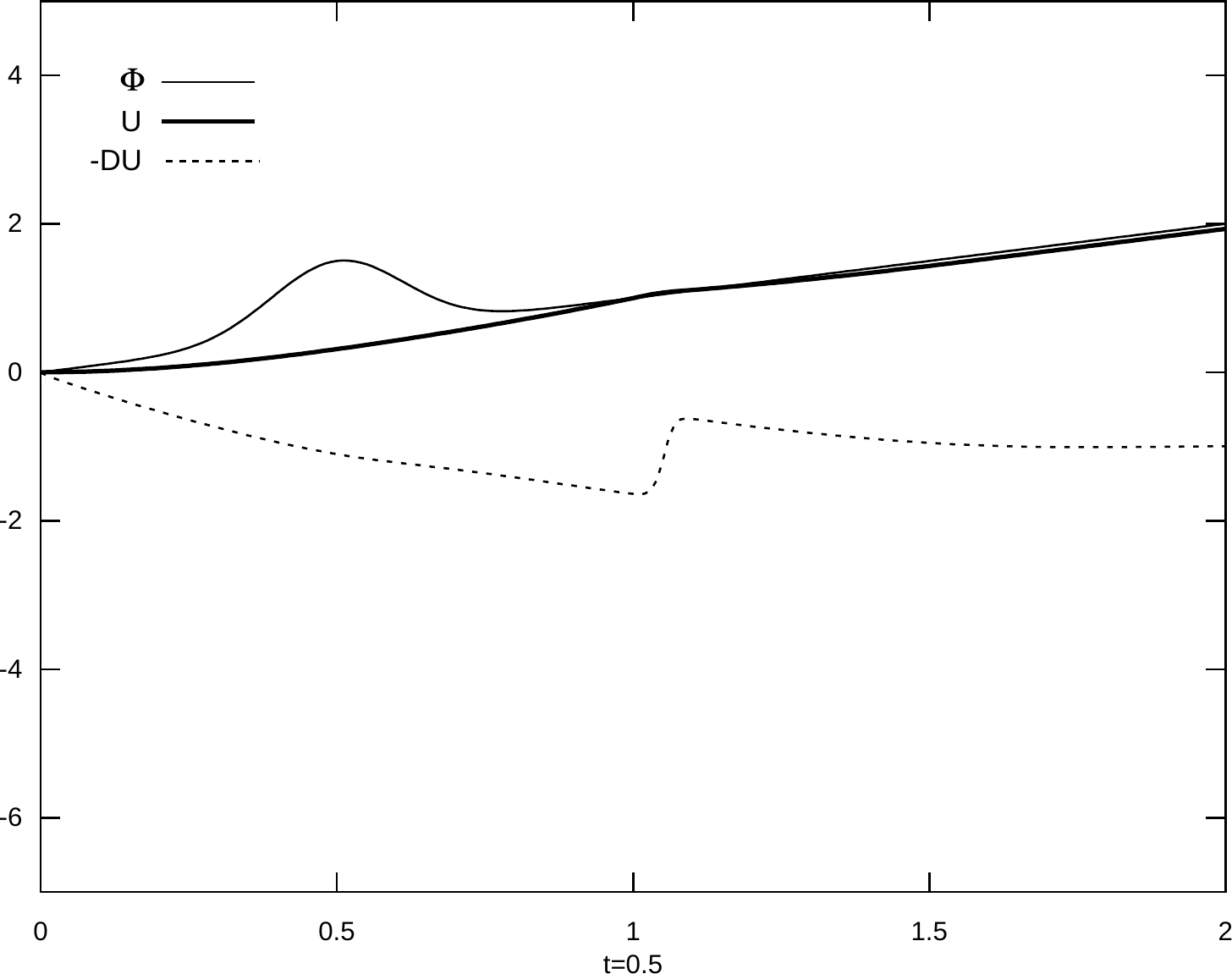}
\caption{Kink generation from a smooth exit cost, value function and optimal control at different times.}\label{kink-gen}
\end{figure}

In Figure \ref{Test3-trj}, we compare the corresponding optimal trajectories, obtained for the initial data $x=0.48$ and $x=0.52$, respectively slightly below and above the barrier. 
\begin{figure}[!h]
\begin{tabular}{cc}
	\includegraphics[width=0.45\textwidth]{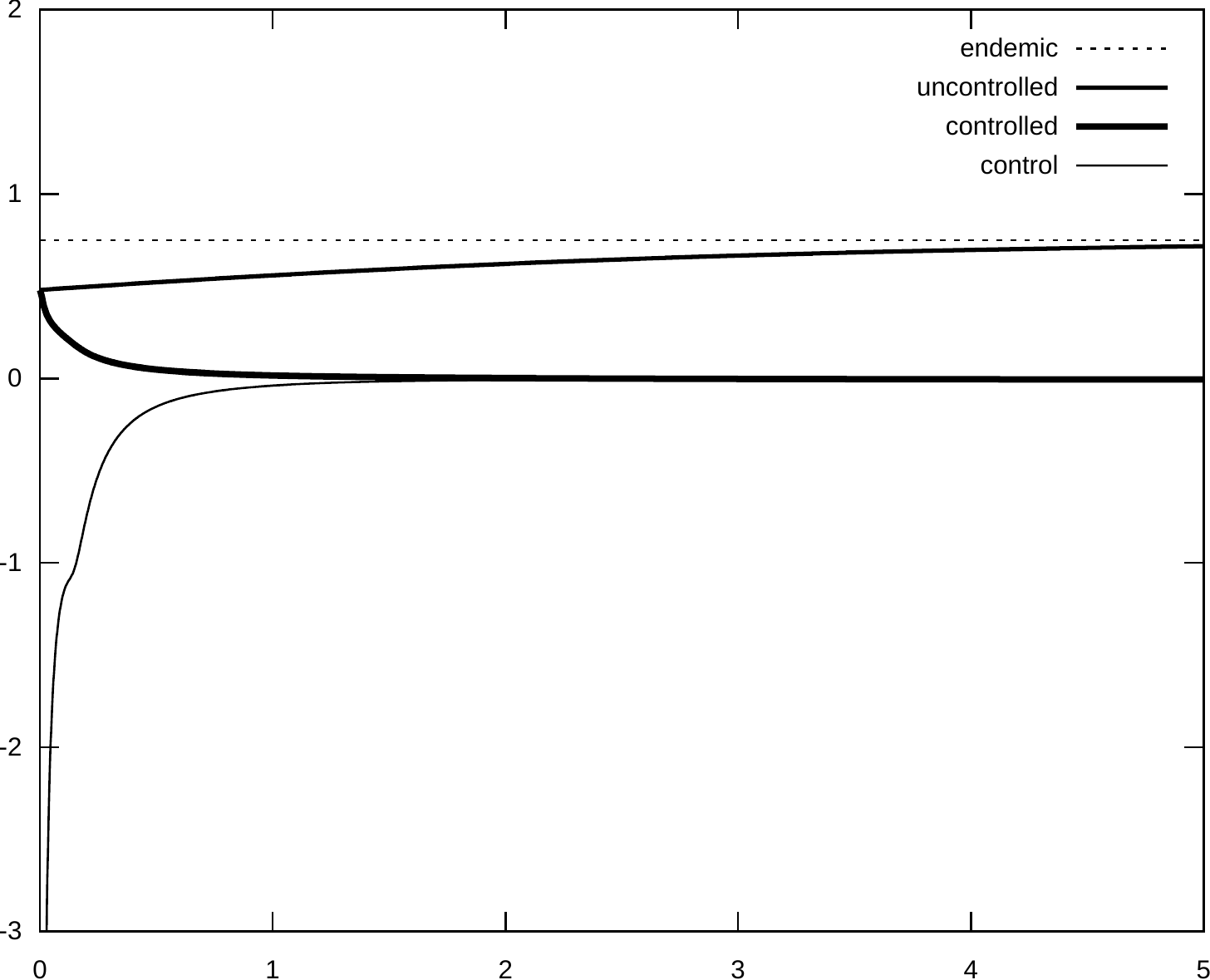} 
	&
	\includegraphics[width=0.45\textwidth]{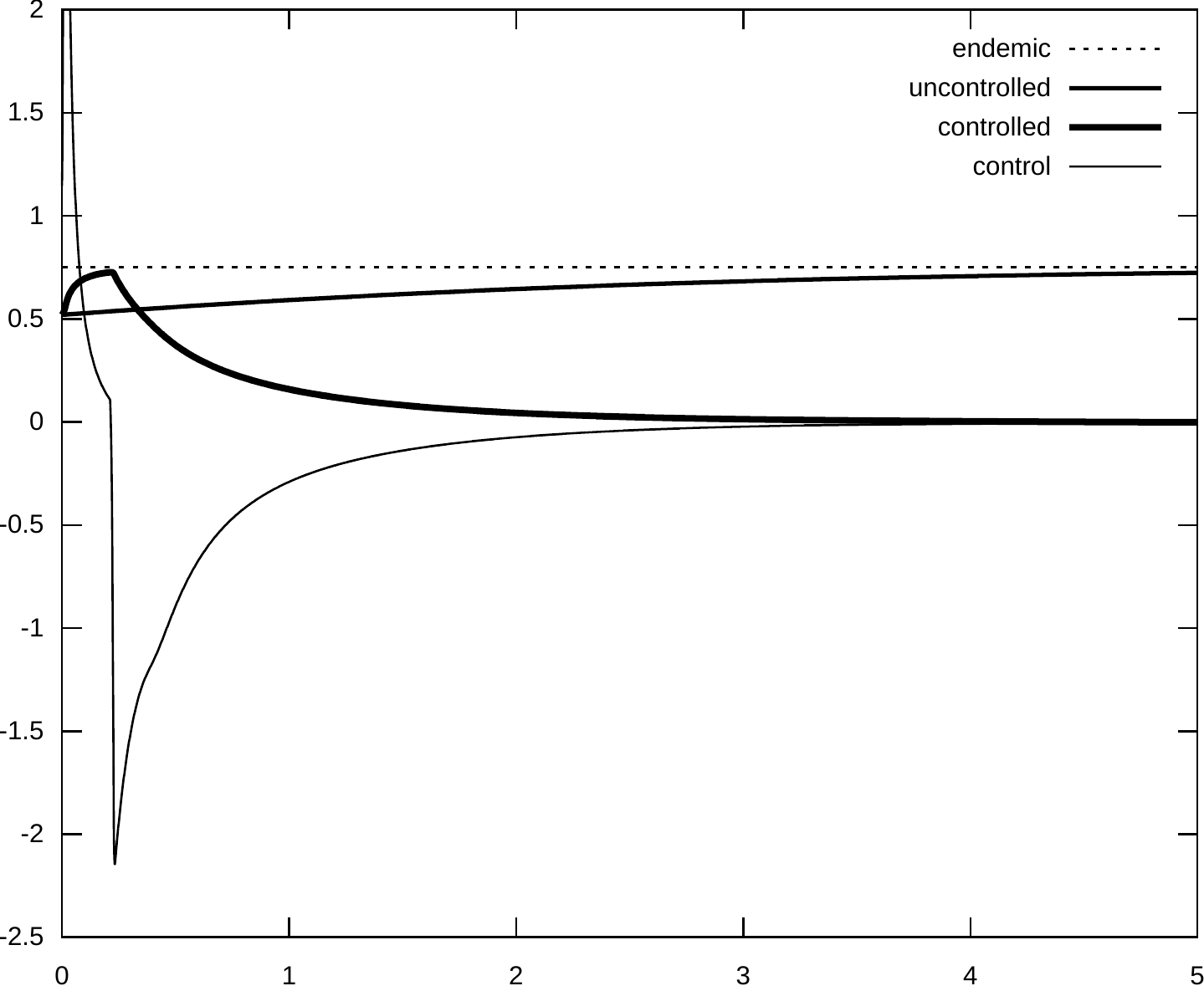}\\
	$x=0.48$ & $x=0.52$
	\end{tabular}
\caption{Optimal trajectories for different initial data $x$.}\label{Test3-trj}
\end{figure}
In the first case, we obtain a controlled trajectory similar to the previous tests, with just a larger amplitude in the control due to the choice of $\phi$. On the other hand, the second case confirms the scenario discussed above. Indeed, the optimal control acts in the positive direction for a small amount of time, pushing the controlled trajectory close to the endemic population, then readily jumps to a negative value, and starts steering the system to the origin.       

To conclude this section, we compare the stationary regime of the solution $u$ of the Hamilton-Jacobi equation \eqref{hj} with the smooth viscosity solution $\bar v^\alpha$ for the stationary equation \eqref{hjstat}, provided in Remark \ref{rmk1} by 
$$\bar v^\alpha(x)=\phi(0)+\int_0^x b_\alpha(s)+\sqrt{b_\alpha^2(s)+s^2}ds.$$
We compute $\bar v^\alpha$ on our numerical grid, approximating the integral by a simple trapezoidal quadrature rule, using the same space step $\Delta x$. Moreover, we set $x_{\max}=T=10$, and we choose the same exit cost $\phi$ of the previous test (note that this affects the convergence of $u$ in time, while $\bar v^\alpha$ only depends on $\phi(0)$). In Table \ref{table}, we report the results of the comparison under grid refinement, for different choices of $\alpha$ and $\rho$, evaluating the difference $u(\cdot,T)-\bar v^\alpha(\cdot)$ both in $L^\infty$ and $L^2$ space norms.  As $\Delta x\to 0$, we clearly observe a decay of the errors, respectively of order $\mathcal{O}(\Delta x)$ and $\mathcal{O}(\Delta x^2)$ for the two norms, and also a slowdown in convergence  as $\alpha$ and $\rho$ decrease. 
This numerical experiment is in agreement with the result proved in Theorem \ref{thm1}. In particular, convergence is obtained on bounded space intervals, and we have observed in all the experiments that possible irregularities of $u$ are pushed out of the domain towards infinity, before approaching the stationary smooth solution $\bar v^\alpha$ in the limit $T\to +\infty$. 
\begin{table}[!h]
    \centering
    \resizebox{\columnwidth}{!}{\begin{tabular}{r|c|c|c|c|c|c|c|c|}
    & \multicolumn{2}{|c|}{\small $\alpha=1$, $\rho=\frac32$} 
    & \multicolumn{2}{|c|}{\small $\alpha=\frac12$, $\rho=\frac32$} 
    & \multicolumn{2}{|c|}{\small $\alpha=1$, $\rho=\frac12$} 
    & \multicolumn{2}{|c|}{\small $\alpha=\frac12$, $\rho=\frac12$}\\ 
    \hline
    $\Delta x$ 
    & \small $L^\infty$ err & \small $L^2$ err   
    & \small $L^\infty$ err & \small $L^2$ err   
    & \small $L^\infty$ err & \small $L^2$ err   
    & \small $L^\infty$ err & \small $L^2$ err\\
    \hline
    
   0.1     & 0.047 & 0.01758 & 0.334 & 0.40162 & 0.089 & 0.04607 & 0.341 & 0.41196 \\\hline
0.05    & 0.023 & 0.00439 & 0.167 & 0.09971 & 0.044 & 0.01147 & 0.170 & 0.10226\\\hline
0.025   & 0.012 & 0.00109 & 0.083 & 0.02484 & 0.022 & 0.00286 & 0.085 & 0.02547\\\hline
0.0125  & 0.006 & 0.00027 & 0.042 & 0.00619 & 0.011 & 0.00071 & 0.043 & 0.00636\\\hline
0.00625 & 0.003 & 0.00007 & 0.021 & 0.00155 & 0.005 & 0.00018 & 0.021 & 0.00158\\
    \hline
            \end{tabular}}
    \caption{Comparison between $u$ and $\bar v^\alpha$ under grid refinement for different model parameters.}
    \label{table}
\end{table}

\bibliographystyle{alpha}
{\small
\bibliography{opthj}}

\end{document}